%% file: center-action-fs.tex
\pdfoutput=1
\documentclass[10pt]{amsart}
\input{sty-center-action-fs.tex}
\DeclareRobustCommand{\gobblefive}[5]{}
\newcommand*{\SkipTocEntry}{\addtocontents{toc}{\gobblefive}}%
\makeatletter
\def\subsection{\@startsection{subsection}{2}%
  \z@{.5\linespacing\@plus.7\linespacing}{.3\linespacing}%
  {\normalfont\bfseries}}
\makeatother
\usepackage{fullpage}
\usepackage[normalem]{ulem}
\author{Konrad Zou}
\title{On the action of the center in the \texorpdfstring{\(\ell\)}{l}-adic categorical Langlands program}
\begin{document}	
	\begin{abstract}
		We discuss the action of the center in the categorical local Langlands program in the form of \cite{geometrization}.
        We provide a spectral description of the convolution action of a central torus in this situation.
        We show that the veracity of the \(\ell\)-adic categorical Langlands conjecture passes to the quotient by a central torus \(D\) satisfying a minor cohomological condition in characteristic 0 and we show that under restriction to Langlands-Shahidi type parameters the conjecture passes to such a quotient in all characterstics and even integrally.
        In that situation we also show that the splitting of the semi-orthogonal decomposition descends.
        Finally we use these results to extend \cite{zou2025categoricallocallanglandsmathrmgln} to \(\PGL_n\).
	\end{abstract}
	\maketitle
	\tableofcontents
    \include{introduction}
    \include{spectral-preprations}
    \include{geometric-preparations}
    \include{spectral-description-shriek-convolution}
    \include{general-compatibility-results}
    \include{compatibility-langlands-shahidi}
    \include{application-pgln}
	\printbibliography%
\end{document}

%% file: sty-center-action-fs.tex
\usepackage{comment}
\usepackage{amsmath,amssymb,amsthm,mathtools}
\usepackage{tikz-cd}
\usepackage{tikz}
\usepackage[utf8]{inputenc}
\usepackage[british]{babel}
\usepackage{lmodern}
\usepackage[T1]{fontenc}
\usepackage[babel=true]{microtype}
\usepackage[mathscr]{euscript}
\usepackage{hyperref}
\usepackage{zref-clever}
\zcsetup{S,nameinlink=false}
\newcommand{\cref}[1]{\zcref{#1}}
\newcommand{\patchzref}[3]{\AddToHook{env/#1/begin}{\zcsetup{countertype={#2=#3}}}}
\newcommand{\patchzrefthmcounter}[2]{\patchzref{#1}{theorem}{#2}}
\zcRefTypeSetup{conjecture}{
Name-sg = Conjecture ,
name-sg = conjecture ,
Name-pl = Conjectures ,
name-pl = conjectures ,
}
\AddToHook{env/theorem/begin}{\zcsetup{countertype={theorem=theorem}}}
\AddToHook{env/lemma/begin}{\zcsetup{countertype={theorem=lemma}}}
\patchzrefthmcounter{corollary}{corollary}
\patchzrefthmcounter{proposition}{proposition}
\patchzrefthmcounter{remark}{remark}
\patchzrefthmcounter{conjecture}{conjecture}
\usepackage[backend=biber, bibencoding=utf8,style=alphabetic,maxnames=99,maxalphanames=99]{biblatex}
\usepackage{relsize}

\DefineBibliographyStrings{english}{%
	bibliography = {References},
}

\binoppenalty=\maxdimen
\relpenalty=\maxdimen

\theoremstyle{plain}
\newtheorem{theorem}{Theorem}[section]
\newtheorem{lemma}[theorem]{Lemma}
\newtheorem{corollary}[theorem]{Corollary}
\newtheorem{proposition}[theorem]{Proposition}

\newtheorem{conjecture}[theorem]{Conjecture}

\theoremstyle{definition}
\newtheorem{definition}[theorem]{Definition}

\newtheorem{remark}[theorem]{Remark}

\theoremstyle{plain} %
\newcommand{\thistheoremname}{}
\newtheorem{genericthm}[theorem]{\thistheoremname}

\theoremstyle{remark}

\DeclareMathOperator{\Hom}{Hom}

\newcommand*{\from}{\colon}
\newcommand*{\defword}[1]{\emph{#1}}

\newcommand*{\defined}{\coloneqq}

\mathchardef\mhyphen="2D
\newcommand*{\xto}[1]{\xrightarrow{#1}}
\DeclareMathOperator*{\colim}{colim}

\DeclareMathOperator{\id}{id}

\DeclareMathOperator{\pr}{pr}

\makeatletter
\def\mydefbb#1{\expandafter\def\csname bb#1\endcsname{\mathbb{#1}}}
\def\mydefallbb#1{\ifx#1\mydefallbb\else\mydefbb#1\expandafter\mydefallbb\fi}
\mydefallbb ABCDEFGHIJKLMNOPQRSTUVWXYZ\mydefallbb
\makeatother
\newcommand*{\injto}{\hookrightarrow}

\newcommand*{\Mod}{\mathrm{Mod}}
\DeclareMathOperator{\Spa}{Spa}

\DeclareMathOperator{\Spec}{Spec}

\DeclareMathOperator{\Fun}{Fun}

\DeclareMathOperator{\An}{Ani}

\DeclareMathOperator{\Rep}{Rep}

\DeclareMathOperator{\Bun}{Bun}

\newcommand*{\Hck}{\mathrm{Hck}}

\newcommand*{\Div}{\mathrm{Div}}
\DeclareMathOperator{\Spd}{\mathrm{Spd}}

\newcommand*{\oo}{\mathcal{O}}

\newcommand*{\localhck}{\mathcal{H}\mathrm{ck}}

\newcommand{\locsys}[2]{Z^1(#1,#2)/#2}
\newcommand{\locsyscoarse}[2]{Z^1(#1,#2)/\!\!/#2}
\newcommand{\locsyssub}[3]{Z^1_{#3}(#1,#2)/#2}
\newcommand{\locsyscoarsesub}[3]{Z^1_{#3}(#1,#2)/\!\!/#2}
\DeclareMathOperator{\cind}{c-ind}
\DeclareMathOperator{\Perf}{Perf}
\DeclareMathOperator{\QCoh}{QCoh}
\DeclareMathOperator{\Coh}{Coh}
\newcommand*{\nilp}{{\mathrm{Nilp}}}
\DeclareMathOperator{\Sing}{Sing}

\newcommand{\op}{{\mathrm{op}}}

\DeclareMathOperator{\End}{End}

\DeclareMathOperator{\D}{\mathcal{D}}

\newcommand{\qc}{\mathrm{qc}}
\DeclareMathOperator{\Exc}{Exc}

\DeclareMathOperator{\ind}{Ind}

\newcommand{\ula}{{\mathrm{ULA}}}

\usepackage{scalerel, stackengine}
\stackMath
\newcommand\widecheck[1]{%
\savestack{\tmpbox}{\stretchto{%
  \scaleto{%
    \scalerel*[\widthof{\ensuremath{#1}}]{\kern-.6pt\bigwedge\kern-.6pt}%
    {\rule[-\textheight/2]{1ex}{\textheight}}%
  }{\textheight}%
}{0.5ex}}%
\stackon[1pt]{#1}{\scalebox{-1}{\tmpbox}}%
}

\newcommand{\ad}{{\mathrm{ad}}}

\DeclareMathOperator{\CAlg}{CAlg}

\newcommand{\irred}{{\mathrm{irred}}}

\newcommand{\GL}{\mathrm{GL}}

\newcommand{\alg}{{\mathrm{alg}}}

\DeclareMathOperator{\Lie}{Lie}

\DeclareMathOperator*{\Ind}{Ind}

\newcommand{\LSt}{{\mathrm{LSt}}}
\newcommand{\wLSt}{{\mathrm{wLSt}}}
\DeclareMathOperator{\Gest}{Gest}
\DeclareMathOperator{\Stk}{Stk}
\DeclareMathOperator{\vStk}{vStk}
\DeclareMathOperator{\PGL}{PGL}
\DeclareMathOperator{\SL}{SL}
\DeclareMathOperator{\const}{const}
\DeclareMathOperator{\IndCoh}{ICoh}
\newcommand{\temp}{{\mathrm{temp}}}

%% file: introduction.tex
\section{Introduction}
Let \(E\) be a non-archimedian local field and \(G\) be a connnected reductive group over \(E\).
For notational simplicity in the introduction, let us also assume that \(G\) is split.
The local Langlands program over a non-archimedian field \(E\) aims to give a finite to one map from the set of irreducible representations of \(G(E)\) in \(\overline{\bbQ_\ell}\), towards the set of continuous homomorphisms \(W_E\to\check{G}(\overline{\bbQ_\ell})\), where \(W_E\) is the Weil group of \(E\) and \(\check{G}\) is the Langlands dual group of \(G\).

The easiest situation is when \(G=\bbG_m\) is commutative, in this case \(\check{G}=\bbG_m\) as well.
It follows that any continuous homomorphism \(W_E\to\check{G}(\overline{\bbQ_\ell})\) factors over the abelianization map \(W_E\to W_E^{\mathrm{ab}}\).
By local class field theory, \(W_E^{\mathrm{ab}}=E^\times\).
Therefore, on the one hand, an irreducible continuous representation of \(E^\times\) is nothing but a \(\overline{\bbQ_\ell}\)-valued character of \(E^\times\), which by local class field theory is also nothing else than a continuous group homomorphism \(W_E\to\overline{\bbQ}_\ell^\times\).
This easily generalized to a product of \(\bbG_m\).

Assume we have some local Langlands correspondence for a group \(G\), and we fix a central torus \(D\subset G\).
Then by Schur's lemma, an irreducible representation \(\pi\) of \(G\) gives rise to a character \(\chi\) of \(D(E)\).
On the dual side, we have a parameter \(\varphi_\pi\from W_E\to\check{G}(\overline{\bbQ_\ell})\).
We may post-compose this with the map \(\check{G}\to\check{D}\) to obtain a map \(\check{\chi}\from W_E\to\check{D}(\overline{\bbQ_\ell})\).
Compatibility with central characters means that \(\chi\) matches up with \(\check{\chi}\) under the local Langlands correspondence for tori.

Recently, in \cite{geometrization}, a semi-simplified version of the local Langlands correspondence has been achieved, and it is compatible with central characters in the above sense, this is \cite[Theorem IX.0.5.(iii)]{geometrization}.
In fact, they formulate a categorical conjecture, and in this setting one can say more.

Instead of considering the set of irreducible representations of \(G(E)\), we consider the category of sheaves on \(\D(\Bun_G)\).
Here \(\Bun_G\) is a stack parametrizing \(G\)-bundles on the Fargues-Fontaine curve.
Instead of the set of continuous group homomorphisms \(W_E\to\check{G}(\overline{\bbQ_\ell})\), we observe that these naturally define a moduli space \(\locsys{W_E}{\check{G}}\) and we consider coherent sheaves on it.
Instead of a finite to one map, here we really want an equivalence of categories.

The easiest case is \(G=\bbG_m\) again, where this was proven in \cite{categorical-fargues-tori}.
Let now \(D\) be a torus as above.
Then we have an action of \(\Bun_D\) on \(\Bun_G\), which takes a \(D\)-bundle \(\mathcal{E}\) and a \(G\)-bundle \(\mathcal{F}\) and maps it to the \(G\)-bundle \(\mathcal{E}\boxtimes\mathcal{F}\times^{D\times G}G\), where the map \(D\times G \to G\) is the multiplication map, which is a group homomorphism as we assumed \(D\) to be central.
This induces an action of \(\D(\Bun_D)\) on \(\D(\Bun_G)\).
Explicitly, writing \(a\from\Bun_D\times\Bun_G\to\Bun_G\) for the action map, consider the diagram
\begin{equation*}
    \begin{tikzcd}
       & \Bun_D\times\Bun_G \arrow[rd, "a"] \arrow[r, "\pr_1"] \arrow[ld, "\pr_2"'] & \Bun_D \\
\Bun_G &                                                                            & \Bun_G
\end{tikzcd}
\end{equation*}
Then the action of some \(A\in\D(\Bun_D)\) on some \(M\in\D(\Bun_G)\) is given by 
\begin{equation*}
    A\star M\defined a_!(\pr_1^*A\otimes\pr_2^*M).
\end{equation*}
The unit of this action is given by \(\mathcal{W}_D\defined e_!\overline{\bbQ_\ell}\), where \(e\from *\to\Bun_D\) selects the trivial \(D\)-bundle.
We can equip \(\D(\Bun_D)\) with the convolution monoidal structure \(-\star-\), so that \(\D(\Bun_G)\) becomes a module over \(\D(\Bun_D)\).
Explicitly, there is a multiplication map \(m\from\Bun_D\times\Bun_D\to\Bun_D\) since \(D\) is commutative, then the monoidal structure is given by \(A\star B\defined m_!(A\boxtimes B)\) and the unit is \(e_!\overline{\bbQ_\ell}\).

On the dual side, we have a morphism of stacks \(\locsys{W_E}{\check{G}}\to\locsys{W_E}{\check{D}}\), this induces an action of \(\Ind\Perf^\qc(\locsys{W_E}{\check{D}})\) on \(\Ind\Coh^\qc(\locsys{W_E}{\check{G}})\).

Again one expects this to match up somehow, and the one of the aims of this article is to discuss how, in the case when \(D\) satisfies \(H^1(\Gamma,\check{D})=0\), also allowing for modular and integral coefficients.
The first important result is the following.
\begin{theorem}[{\cref{lem: cartier dual BunT QCoh(ParT)}}]
    Assume \(H^1(\Gamma,\check{D})=0\). There is an equivalence of symmetric monoidal categories between \((\D(\Bun_D),\star)\) and \(\QCoh(\locsys{W_E}{\check{D}})\).
\end{theorem}
This is extremely explicit, one can deduce this result in general by resolving \(D\) by induced tori, which satisfy the cohomological condition by Hilbert 90, however this generality is not needed here.
Having this result at hand, there are now two different actions of \(\QCoh(\locsys{W_E}{\check{D}})\) on \(\D(\Bun_G)\).
One is via the equivalence \(\QCoh(\locsys{W_E}{\check{D}})\simeq(\D(\Bun_D),\star)\) and using the convolution action.
Another way is to observe that we have a symmetric monoidal functor \(\QCoh(\locsys{W_E}{\check{D}})\to\Ind\Perf(\locsys{W_E}{\check{G}})\), via the composite
\begin{equation*}
    \QCoh(\locsys{W_E}{\check{D}})=\Ind\Perf^\qc(\locsys{W_E}{\check{D}})\to\Ind\Perf(\locsys{W_E}{\check{D}})\to\Ind\Perf(\locsys{W_E}{\check{G}})
\end{equation*}
where the last map is pulling back along \(\locsys{W_E}{\check{G}}\to\locsys{W_E}{\check{D}}\).
It turns out that these are the same.
\begin{theorem}[\cref{thm: geometric convolution action is spectral action}]
    Assume that \(D\) is an satisfies \(H^1(\Gamma,D)=0\). 
    The two actions of \(\QCoh(\locsys{W_E}{\check{D}})\) on \(\D(\Bun_G)\) agree.
\end{theorem}
For such tori, this is easy as one can treat certain line bundles on \(\locsys{W_E}{\check{D}}\) and the action of \(R\Gamma(\locsys{W_E}{\check{D}},\oo)\) separately from each other. %

In this case, one can wonder what compatibility results we can obtain for the general categorical conjecture.
Let us fix a quasi-split reductive group \(G\) over \(E\), a central torus \(D\subset G\) and set \(G'\defined G/D\).
Let \(\Gamma\defined\mathrm{Gal}(\overline{E}/E)\).
Let us first recall the full categorical conjecture.
\begin{conjecture}[{\cite[Conjecture X.3.5.]{geometrization}}]\label{conj: categorical conjecture}
    Assume that \(\ell\nmid\lvert\pi_0(Z(G))\rvert\).
    Fix a Whittaker datum \((U,\psi)\), giving rise to a sheaf \(\mathcal{W}\defined i_{1!}\cind_{U(E)}^{G(E)}\psi\).
    Consider the functor
    \begin{equation*}
        -*\mathcal{W}\from\Ind\Perf^\qc(\locsys{W_E}{\check{G}})\to\D(\Bun_G)
    \end{equation*}
    given by spectral action on \(\mathcal{W}\).
    Then we have:
    \begin{enumerate}
        \item The right adjoint \(c_{\mathcal{W}}\) of \(-*\mathcal{W}\) is fully faithful on \(\D(\Bun_G)^\omega\).
        \item The essential image of \(\D(\Bun_G)^\omega\) under \(c_{\mathcal{W}}\) is given by \(\Coh_{\nilp}(\locsys{W_E}{\check{G}})\).
    \end{enumerate}
\end{conjecture}
\begin{remark}
    The statement of the conjecture depends on a Whittaker datum \((U,\psi)\).
    Of course conjugate Whittaker data give equivalent conjectures, however in general not all Whittaker data are conjugate.
    In particular, a priori it could be possible that the conjecture holds for one Whittaker datum but not for another.\footnote{Of course we do not expect this to be the case.}
    There are two ways of resolving this issue.
    One is via \cite[Conjecture 5.2.1.]{hansen-mann}.
    Another is upcoming work of Naoki Imai and Laurent Fargues which will formulate variants of the \(\ell\)-adic categorical conjecture without appealing to Whittaker data.
    In particular if any of these two conjectures holds, then \cref{conj: langlands shahidi type conjecture} holds for all Whittaker data.
\end{remark}
\begin{remark}
    Assume \(\ell\in\Lambda^\times\), then the formulation of the conjecture simplifies on the following ways.
    First, we have an equivalence of presentably symmetric monoidal categories \(\Ind\Perf^\qc(\locsys{W_E}{\check{G}})\simeq\QCoh(\locsys{W_E}{\check{G}})\).
    Second, the condition \(\nilp\) becomes automatic, so that \(\Coh_\nilp(\locsys{W_E}{\check{G}})=\Coh(\locsys{W_E}{\check{G}})\).
\end{remark}
We get the following compatibility.
\begin{theorem}[{\cref{thm: compatiblity full categorical conjecture}}]
    Assume that \(D\) satisfies \(H^1(\Gamma,D)=0\) and all primes that are not banal for \(D\) are invertible in \(\Lambda\).
    Assume that \cref{conj: categorical conjecture}(1) holds for \(G\) for a Whittaker datum \((U,\psi)\).
    This induces a Whittaker datum for \(G'\) since there is a canonical bijection between Borel subgroups of \(G\) and \(G'\). \cref{conj: categorical conjecture}(2) holds for \(G'\) for this induced Whittaker datum.
    If \(\ell\in\Lambda^\times\), then \cref{conj: categorical conjecture}(2) for \(G\) implies \cref{conj: categorical conjecture} for \(G'\).
\end{theorem}
\begin{remark}
    We believe the condition on non-banal primes to be an artifact of the proof strategy.
\end{remark}
Linus Hamman found the condition of (weakly) Langlands-Shahidi type.
For a precise discussion, we refer to \cref{sec: (weakly) Langlands-Shahidi type}.
Let us remark that for \(G=\GL_n\), weakly Langlands-Shahidi type means that for a direct sum decomposition \(\varphi\cong\varphi_1\oplus\dots\varphi_r\), we have \(\varphi_i\ncong\varphi_j(1)\) for \(i\neq j\), Langlands-Shahidi type means we additionally require \(\varphi_i\ncong\varphi_j\) for \(i\neq j\) as well.
We write \(\D^\LSt(\Bun_G)\) for sheaves on \(\Bun_G\) with parameters of Langlands-Shahidi type and \(\D^\wLSt(\Bun_G)\) for sheaves on \(\Bun_G\) of weakly Langlands-Shahidi type.
In this case, one can conjecture more.
For one, there is a semi-orthogonal decomposition on \(\D^\LSt(\Bun_G)\) induced by \(\lvert\Bun_G\rvert\).
\begin{conjecture}[{\cite[Conjecture 3.2.17.]{hamann-thesis}}]
    The semi-orthogonal decomposition on \(\D^\LSt(\Bun_G)\) splits, i.e.
    \begin{equation*}
        \D^\LSt(\Bun_G)\simeq\prod_{b\in\lvert\Bun_G\rvert}\D^{\LSt}(\Bun_G^b),
    \end{equation*}
    where the functor is induced by the collection \(i_{b!}\), equivalently \(i_{b*}\) equivalently \(i_{b\sharp}\).
    All three pushforwards are isomorphic in this case.
\end{conjecture}
\begin{theorem}[{\cref{thm: langlands shahidi splitting descends}}]
    If previous conjecture holds for \(G\), then it also holds for \(G'\).
\end{theorem}
There is also a categorical conjecture for \(\D^\wLSt(\Bun_G)\), we refer to \cref{conj: langlands shahidi type conjecture} for a precise formulation.
Observe that there is an additional \(t\)-exactness claim that appears as \cref{conj: langlands shahidi type conjecture}(3).
In this case, we get the expected compatibility claim, this appears as \cref{thm: Langlands-Shahidi type compatibility}.
Note that in this case, no condition on non-banal primes for \(D\) appears.

The main method we use apart from \cref{thm: geometric convolution action is spectral action} is the fact that \(p\from\Bun_G\to\Bun_G'\) is smooth with trival dualizing complex, this is \cref{lem: universal !-descent BunG BunG'}, which allows us to explicitly handle various Lurie tensor product computations, which use \(p^!\), which we can replace by the more explicit \(p^*\).

In the last section we observe that \(G=\GL_n\) is an example were we can apply the whole discussion, using the main results of \cite{zou2025categoricallocallanglandsmathrmgln}.

\SkipTocEntry
\subsection{Acknowledgments}
We thank Peter Scholze for pointing out that to compute tensor products for convolution actions, one should prove \(!\)-descent.
We thank Jean-Fran\c{c}ois Dat for explain to us what it means to be compatible with central characters in the classical local Langlands correspondence.
The author has been funded by the Simons Collaboration on Perfection in Algebra, Geometry, and Topology as postdoctoral researcher at the CNRS.
\SkipTocEntry
\subsection{Notation}
\leavevmode
\begin{itemize}
    \item Let \(E\) be a nonarchimedian local field of residue characterstic \(p\) and residue field \(\bbF_{q}\), absolute Galois group \(\Gamma\) and \(C\) a completed algebraic closure of \(E\).
    \item Let \(\Lambda\) be a \(\bbZ_\ell[\sqrt{q}]\)-algebra.
    \item Given a connected reductive group \(G\) over \(E\), we write \(\check{G}\) for the Langlands dual group. We write \(\Bun_G\) for the stack of \(G\)-bundles on the Fargues-Fontaine curve and \(\locsys{W_E}{\check{G}}\) for the stack of \(L\)-parameters. %
    \item Let \(\ell\neq p\) be an auxilliary prime satisfying \(\ell\nmid\lvert\pi_0(Z(G))\rvert\). This ensures that the spectral action is well-defined.
    \item All relevant \(v\)-stacks shall live over \(\Spd(\overline{\bbF_q})\) for us, unless mentioned otherwise. We write \(\vStk\) for the category thereof.
    \item We write \(\D\) for \(\D_{\mathrm{mot}}(-)\otimes_{\D_{\mathrm{mot}}(\Spd(\overline{\bbF_q}))}\Mod_\Lambda\), where the map \(\D_{\mathrm{mot}}(\overline{\bbF_q})\to\Mod_\Lambda\) is induced by the \'etale realization of motives as explained in \cite{motivic-geometrization}.
    \item We write \(\Bun_G\) for the stack of \(G\)-bundles on the Fargues-Fontaine curve.
    \item For an Artin \(v\)-stack \(X\) let \(\D(X,\Lambda)\) denote the category of lisse sheaves as defined in \cite[Definition VII.6.1.]{geometrization}, sometimes we will just write \(\D(X)\).
    
    \item For a connected reductive group \(H/E\), let \(\D(H(E),\Lambda)\) denote the derived category of smooth \(H(E)\)-representations, sometimes we will just write \(\D(H(E))\).

    \item For \(b\in B(H)\), let \(i_b\from\Bun_H^b\to\Bun_H\) be the canonical locally closed immersion and \(s_b\from [*/H_b(E)]\to\Bun_H^b\) be the splitting of \(\Bun_H^b\to [*/H_b(E)]\) as constructed in \cite[Proposition III.5.3.]{geometrization}.
    
    \item We will work fully derived, all functors are implicitly derived.
    \item We will say ``category'' for \((\infty,1)\)-category.
    \item We write \(\Pr\) for the category of \(\aleph_1\)-presentable categories. Since we will not work with other presentable categories, we will just say ``presentable category'' to mean ``\(\aleph_1\)-presentable category''.\footnote{This is only used in \cref{lem: 2QCoh(BD)}, as we rely on the theory of Gestalten which uses \(\aleph_1\)-presentable categories. The difference between using them and all presentable categories is immaterial in this article.}
    \item If we write a tensor product of categories, we mean the tensor product on \(\Pr\).
    \item By ``ring'' we mean ``commutative ring''.
    \item We write \(\An\) for the category of anima, also called \(\infty\)-groupoids in the literature.
    \item We write \(\Stk^\alg=\Fun(\mathrm{CRing},\An)\), where \(\mathrm{CRing}\) is the category of commutative rings.
    \item On \(\Stk^\alg\) we consider \(\QCoh\) via right Kan extension of the assignment \(A\mapsto\Mod_A\). We consider this as a 6-functor formalism by taking all morphisms to be in \(P\) on \(\mathrm{CRing}\) and only isomorphisms to be in \(I\). This defines a 6-functor formalisms on affine schemes by \cite[Proposition A.5.10]{mann2022padic6functorformalismrigidanalytic}, which we extend to all stacks using \cite[Theorem 3.4.11]{heyer-mann}.
    \item For \(X\) an algebraic stack, we write \(\Perf^\qc(X)\) for the category of perfect complexes with quasi-compact support.
    \item In regards to \(t\)-structures, we use subscript to denote homological conventions and superscripts do denote cohomological conventions. This means that if \(\mathcal{C}\) is a stable category with \(t\)-structure, \(\mathcal{C}_{\geq 0}=\mathcal{C}^{\leq 0}\), and \(\mathcal{C}_{\geq 0}\) denotes connective objects.
\end{itemize}

%% file: spectral-preprations.tex
\section{Spectral preparations}
All geometry in this section will live over \(\Lambda\), although the results hold for any ring, even any \(\mathrm{E}_\infty\)-ring spectrum.
The goal is to understand the datum of a \(\QCoh(\locsys{W_E}{\check{T}})\), at least when \(T\) is an satisifies \(H^1(\Gamma,T)=0\).
The following is well-known to the experts and probably admits a much more elementary proof.
\begin{lemma}\label{lem: 2QCoh(BD)}
    Let \(M\) be a finitely generated abelian group, let \(D=\Hom(M,\bbG_m)\) be the diagonalizable group scheme attached to \(M\).
    Then \(\Mod_{\QCoh(BD)}(\Pr)\) identifies with \(\Fun(BM,\Pr)\).
\end{lemma}
\begin{proof}
    This is immedate from \cite[Theorem 10.8]{gestalten}.
    As \(BM=\colim_{n\in\Delta^\op} M^n\), even in the category of categories, as all edges of the associated cocartesian fibration are cocartesian and thus get inverted, we see that \(\Fun(BM,\Pr)=\lim\Fun(M^n,\Pr)\).
    This is precisely the formula that also computes \((\oo(BM)/\oo(\Gest(\Lambda)))_2\).
    To see that \(BD\to *\) is 1-affine, it suffices to see that \(D\to *\) is 0-affine, as explained after \cite[Theorem 12.3.]{gestalten}, this is clear.
\end{proof}
\begin{remark}\label{rem: 2QCoh(BD) in the free case explicit description}
    Assume that \(M\) is additionally free, with basis \(m_1,\dots,m_n\).
    We see that a category \(\mathcal{C}\) equipped with an action of \(\QCoh(BD)\) is the same thing as giving a functor \(M\to\End(\mathcal{C})\).
    As \(M\) as assumed to be free on \(m_1,\dots,m_n\), this is equivalent to giving the images of \(m_i\), this equivalent to the datum of \(n\) endofunctors \(F_i\) on \(\mathcal{C}\).
    We stress that there are no other coherence conditions assumed here!
\end{remark}
\begin{lemma}\label{lem: cartier duality on qcoh}
    In the setting of the previous lemma, \(\QCoh(D)\) is monoidally equivalent to \(\QCoh(BM)\) with the convolution monoidal structure.
    Similarly, \(\QCoh(BD)\) is monoidally equivalent to \(\QCoh(M)\) with the convolution monoidal structure.
\end{lemma}
\begin{proof}
    This is \cite[Proposition 2.4.3.]{camargo2025cartierdualitygerbesvector} for the 6-functor formalism \(\QCoh\).
    This 6-functor formalism has categorical K\"unneth, so that the 2-presentable kernel category satisfies \(K_{\QCoh,\Spec(\Lambda)}=\Mod_{\Mod_\Lambda}(\Pr)\), this is \cite[Proposition 9.3.]{gestalten}.
\end{proof}
\begin{corollary}\label{cor: spectral action H1 vanishing torus}
    Let \(M_1,M_2\) be two finitely generated groups, write \(D_i=\Hom(M_i,\bbG_m)\).
    Then \(\Mod_{\QCoh(D_1\times BD_2)}(\Pr)\) identifies with \(\Fun(BM_2,\Mod_{\QCoh(D_2)}(\Pr))\).
\end{corollary}
\begin{proof}
    This follows from \cref{lem: 2QCoh(BD)}, using that 
    \begin{align*}
        \Mod_{\QCoh(D_1\times BD_2)}(\Pr)&=\Mod_{\QCoh(D_1)\otimes\QCoh(BD_2)}(\Pr)\\
        &=\Mod_{\QCoh(D_1)}(\Pr)\otimes_{\Pr}\Mod_{\QCoh(D_2)}(\Pr)\\
        &=\Mod_{\QCoh(D_2)}(\Pr)\otimes_{\Pr}\Fun(BM_2,\Pr)\\
        &=\Fun(BM_2,\Mod_{\QCoh(D_2)}(\Pr)).
    \end{align*}
    The first equality is \cite[Theorem 4.7.]{integral-transforms}, the second equality is  \cite[Proposition 4.1.(2)]{integral-transforms}, the third equality is \cref{lem: 2QCoh(BD)}, and the last equality follows from Yoneda.
\end{proof}
\begin{lemma}\label{lem: splitting gerbe}
    Recall from \cite[Lemma 4.1.4.]{categorical-fargues-tori} the map \(\locsys{W_E}{\check{T}}\to\Hom(T(E),\bbG_m)\), which is a \(\check{T}^\Gamma\)-gerbe.
    This gerbe splits, so we have an isomorphism
    \begin{equation*}
        \locsys{W_E}{\check{T}}\cong\Hom(T(E),\bbG_m)\times B\check{T}^\Gamma
    \end{equation*}
\end{lemma}
\begin{proof}
    Note that under the assumption that \(H^1(\Gamma,T)=0\), the map \(\check{T}^\Gamma\injto\check{T}\) admits a retraction, as \(X^*(\check{T}^\Gamma)\) is a free \(\bbZ\)-module of finite rank.
    Then the maps \(\locsys{W_E}{\check{T}}\to\Hom(T(E),\bbG_m)\) and \(\locsys{W_E}{\check{T}}\to B\check{T}\to B\check{T}^\Gamma\) give rise to a map \(\locsys{W_E}{\check{T}}\to\Hom(T(E),\bbG_m)\times B\check{T}^\Gamma\), which one easily verifies to be an isomorphism.
\end{proof}
\begin{remark}\label{rem: enough to give incoherent actions}
    Let us continue the discussion in \cref{rem: 2QCoh(BD) in the free case explicit description}.
    Let \(M_2\) be a free finitely generated abelian group with basis \(m_1,\dots,m_n\).
    Then giving a presentable category \(\mathcal{C}\) with an action of \(\QCoh(D_1\times BD_2)\) is equivalent to giving a \(\oo(D_2)\)-linear structure on \(\mathcal{C}\) together with \(\oo(D_2)\)-linear endofunctors \(F_i\) on \(\mathcal{C}\) attached to each \(m_i\).
    When \(T\) is an satisfies \(H^1(\Gamma,T)\), then \(\locsys{W_E}{\check{T}}\) is precisely of the form \(D_1\times BD_2\) as before.
\end{remark}

%% file: geometric-preparations.tex
\section{Geometric preparations}
Let \(T\) be a torus, such that \(H^1(\Gamma,T)=0\).
By Nakayama-Tate duality, this is equivalent to \(\pi_0(Z(\check{T})^\Gamma)\) being trivial, equivalently that \(B(T)\) is torsion free.
\begin{lemma}\label{lem: splitting BunT in the induced case}
    We have an identification \([*/T(E)]\times X_*(T)_\Gamma\cong\Bun_T\).
\end{lemma}
\begin{proof}
    As \(B(T)\) is torsion free, this is \cite[Example III.4.6.]{geometrization}.
\end{proof}
\begin{lemma}\label{lem: qcoh and constructibel on anima same}
    Let \(\An\to\Stk^\alg, X\mapsto X^{\alg}\) be constructed by mapping \(*\) to \(\Spec(\Lambda)\) and extending via colimits.
    Then \(\QCoh(X^\alg)\) identifies with \(\D(X\times\Spd(\overline{\bbF_q}))\) as 6-functor formalisms on anima, where we only consider truncated morphisms as \(!\)-able.
\end{lemma}
\begin{proof}
    Both are obtained by passage to stacks from the 6-functor formalism on \(C=*\) that sends \(*\) to \(\Mod_\Lambda\).
\end{proof}
\begin{lemma}\label{lem: cartier duality constructible on M}
    Let \(M\) be a finitely generated abelian group, let \(D=\Hom(M,\bbG_m)\).
    Then \(\D(M\times\Spd(\overline{\bbF_q}))\) with the convolution monoidal structure identifies with \(\QCoh(BD)\) with its usual monoidal structure.
    Similarly, \(\D(B(M\times\Spd(\overline{\bbF_q})))\) with the convolution monoidal structure identifies with \(\QCoh(D)\) with its usual monoidal structure.
\end{lemma}
\begin{proof}
    By \cref{lem: qcoh and constructibel on anima same}, \(\D(M\times\Spd(\overline{\bbF_q}))=\QCoh(M)\), with the convolution monoidal structure.
    The claim is now \cref{lem: cartier duality on qcoh}.
\end{proof}
\begin{theorem}\label{lem: cartier dual BunT QCoh(ParT)}
    The category \(\D(\Bun_T)\) with the convolution monoidal structure identifies with \(\QCoh(\locsys{W_E}{\check{T}})\) and its usual monoidal structure.
    Under this equivalence, \(\D(\Bun_T^0)\) and \(\QCoh(\locsyscoarse{W_E}{\check{T}})\) agree.
    More precisely, each compact open subgroup \(K\subset T(E)\), we get open substacks \(\locsyscoarsesub{W_E}{\check{T}}{K}\) and we have an equivalence of symmetric monoidal categories
    \begin{equation*}
        \QCoh(\locsyscoarse{W_E}{\check{T}})=\lim_{K\subset T(E)}\QCoh(\locsyscoarsesub{W_E}{\check{T}}{K})\simeq\lim_{K\subset T(E)}\D([*/(T(E)/K)])=\D(\Bun_T^0).
    \end{equation*}
    In addition, this equivalence is \(t\)-exact.
\end{theorem}
\begin{proof}
    We may base change \(\Bun_T\) to \(\Spa(C)\) without changing the category of sheaves.
    In addition, as in the proof of \cref{lem: splitting BunT in the induced case}, the problem splits into a \([*/T(E)]\) and a \(X_*(T)_{\Gamma}\) part.
    On the other side, \(\locsys{W_E}{\check{\bbG_m}}=\Hom(T(E),\bbG_m)\times B\check{T}^\Gamma\) by \cref{lem: splitting gerbe}.
    That \(\D(X_*(T)_{\Gamma}\times\Spa(C))\) matches up with \(\QCoh(B\check{T}^\Gamma)\) is \cref{lem: cartier duality constructible on M}.
    We are left with comparing \(\D([*/T(E)])\) and \(\QCoh(\Hom(T(E),\bbG_m))\).
    This follows from \cref{lem: cartier duality constructible on M}. 

    For the claim about \(t\)-structures, it suffices to check that the equivalence
    \begin{equation*}
        \QCoh(\locsyscoarsesub{W_E}{\check{T}}{K})\simeq\D([*/(T(E)/K)])
    \end{equation*}
    is \(t\)-exact, as the transition maps induced from changing \(K\) are \(t\)-exact on both sides.
    On the left hand side, the connective part is generated by the structure sheaf \(\oo\), on the left hand side by the \(T(E)/K\)-representation \(\cind_K^{T(E)}\Lambda\).
    The equivalence of these two categories matches these two objects, so the equivalence is also \(t\)-exact.
\end{proof}

%% file: spectral-description-shriek-convolution.tex
\section{Spectral description of \texorpdfstring{\(!\)}{!}-convolution}
Let \(G\) be a reductive group over \(E\) and \(D\) be a torus that satisfy \(H^1(\Gamma,D)=0\) that is a subgroup of the center of \(G\).
Let us write \(G'=G/D\).
In this case, we obtain a sequence of stacks
\begin{equation*}
    0\to\Bun_D\to\Bun_G\to\Bun_{G'}\to 0,
\end{equation*}
in particular \(\Bun_D\) acts on \(\Bun_G\).
This action gives rise to an action of \(\D(\Bun_D)\) on \(\D(\Bun_G)\), where we equip \(\D(\Bun_D)\) with the convolution monoidal structure.
Explicitly, consider the diagram
\begin{equation*}
    \begin{tikzcd}
       & \Bun_D\times\Bun_G \arrow[rd, "a"] \arrow[r, "\pr_1"] \arrow[ld, "\pr_2"'] & \Bun_D \\
\Bun_G &                                                                            & \Bun_G
\end{tikzcd}
\end{equation*}
Then the action of some \(A\in\D(\Bun_D)\) on some \(M\in\D(\Bun_G)\) is given by 
\begin{equation*}
    A\star M\defined a_!(\pr_1^*A\otimes\pr_2^*M).
\end{equation*}
The unit of this action is given by \(\mathcal{W}_D\defined e_!\Lambda\), where \(e\from \Spd(\overline{\bbF_q})\to\Bun_D\) selects the trivial \(D\)-bundle.
One computes as in \cite[Proposition 5.4.4.]{heyer-mann}, that this is identified with \(\cind_*^{D(E)}\Lambda\), which is precisely the Gelfand-Graev representation attached to the unique Whittaker datum of \(D\).
By \cref{lem: cartier dual BunT QCoh(ParT)}, this is also an action of \(\QCoh(\locsys{W_E}{\check{D}})\) on \(\Bun_G\).
Let us call this the \defword{geometric \(\QCoh(\locsys{W_E}{\check{D}})\)-action}.%

On the Langlands dual side, we have a morphism of stacks \(\locsys{W_E}{\check{G}}\to\locsys{W_E}{\check{D}}\), thus the spectral action of \(\Ind\Perf(\locsys{W_E}{\check{G}})\) induces an action of \(\Ind\Perf(\locsys{W_E}{\check{D}})\) on \(\D(\Bun_G)\).
Observe that \(\QCoh(\locsys{W_E}{\check{D}})\) is compactly generated by perfect complexes with quasi-compact support, thus we have a monoidal embedding \(\QCoh(\locsys{W_E}{\check{D}})\to \Ind\Perf(\locsys{W_E}{\check{D}})\), giving rise to an action of \(\QCoh(\locsys{W_E}{\check{D}})\) on \(\D(\Bun_G)\).
Let us call this the \defword{spectral \(\QCoh(\locsys{W_E}{\check{D}})\)-action}.%
The goal is the following theorem.
\begin{theorem}\label{thm: geometric convolution action is spectral action}
    The geometric and the spectral \(\QCoh(\locsys{W_E}{\check{D}})\)-action agree.
\end{theorem}
Combining \cref{lem: splitting gerbe} and \cref{rem: enough to give incoherent actions}, the problem of understanding the spectral and geometric action decomposes into understanding the action of \(\QCoh(\Hom(D(F),\bbG_m))\) and the action of \(\QCoh(BD^\Gamma)\) separately.
Let us start with the line bundles in general.

For this, consider a character \(\chi\) of \(\check{D}\) identified with a line bundle on \(B\check{D}\) and let \(\oo(\chi)\) denote the sheaf on \(\QCoh(\locsys{W_E}{\check{D}})\) given by pulling back \(\chi\) along \(\locsys{W_E}{\check{D}}\to B\check{D}\).
Let \(z\from\locsys{W_E}{\check{G}}\to\locsys{W_E}{\check{D}}\) denote the map induced by \(\check{D}\subset\check{G}\).
The desired compatibility is the following.
\begin{lemma}\label{lem: compatible on line bundles}
    There is an equivalence of functors \((\oo(\chi)*\mathcal{W}_D)\star -\) and \(z^*\oo(\chi)*-\) for \(\chi\in\Rep(\check{D}^I)\) for irreducible \(\chi\), functorial in \(\chi\), \(I\) and compatible with the Weil group action.
\end{lemma}
\begin{proof}
    Only to unclutter the notation, we will assume \(I=\{*\}\), in the general case all occuring mentions of \(\Div^1\) need to be replaced with \((\Div^1)^I\), and all Hecke stacks acquire an additional superscript \(I\).
    Consider the following diagrams:
    \begin{equation}
        \begin{tikzcd}
\Bun_G\times\Div^1                                             & \Hck_{ G} \arrow[l, "h_1"'] \arrow[r, "h_2"]                                                 & \Bun_{G}\times\Div^1                                           \\
\Bun_{D\times G}\times\Div^1 \arrow[d, "a"] \arrow[u, "\pr_2"] & \Hck_{D\times G} \arrow[d, "a_H"] \arrow[l, "h_1'"'] \arrow[r, "h_2'"] \arrow[u, "\pr_{2H}"] & \Bun_{D\times G}\times\Div^1 \arrow[d, "a"] \arrow[u, "\pr_2"] \\
\Bun_G\times\Div^1                                             & \Hck_{G} \arrow[l, "h_1"'] \arrow[r, "h_2"]                                                  & \Bun_G\times\Div^1                                            
\end{tikzcd}
    \end{equation}
    and
    \begin{equation}\label{eq: crucial almost cartesian square}
        \begin{tikzcd}
\Bun_G\times\Div^1 \arrow[d, "e'"] & \Hck_{G} \arrow[l] \arrow[r] \arrow[d, "e'_H"] & \Bun_G\times\Div^1 \arrow[d, "e'"] \\
\Bun_{D\times G}\times\Div^1       & \Hck_{D\times G} \arrow[l] \arrow[r]           & {\Bun_{D\times G}\times\Div^1,}   
\end{tikzcd}
    \end{equation}
    where \(e'\) is the map induced by \((e,\id)\from G\to D\times G\), similarly for \(e'_H\).
    Let \(\chi'\) be the representation of \(\check{D}\times\check{G}\) given by \(\chi\boxtimes\mathrm{triv}\).
    We compute:
    \begin{align*}
        (\oo(\chi)*\mathcal{W}_D)\star A&\cong a_! T_{\chi'}(\mathcal{W}_D\otimes\pr_2^*A)\\
        &\cong h_{2!}a_{H!}(h_1'^*(\pr_{2}^*A\otimes e_!\Lambda)\otimes S_{\chi'})\\
        &\cong h_{2!}a_{H!}(h_1'^*e'_!A \otimes S_{\chi'}).
    \end{align*}
    We wish to apply base change to proceed from here, however the left hand square in \cref{eq: crucial almost cartesian square} is not quite cartesian, we have to modify it to the following diagram (implicitly using that \(\Hck_{D\times G}\cong\Hck_{D}\times\Hck_G\)):
    \begin{equation*}
\begin{tikzcd}
\Bun_G\times\Div^1 \arrow[d, "e'"]          & \Hck_{G}^{\leq0} \arrow[l, "h_1^0"'] \arrow[d, "\bar{e}'_H"]                                               &                                             \\
\Bun_{D\times G}\times\Div^1 \arrow[d, "a"] & \Hck_{D}^{\leq\chi}\times\Hck_G^{\leq 0} \arrow[l, "\bar{h}'_1"'] \arrow[r, "\bar{h}'_2"] \arrow[d, "a_H"] & \Bun_{D\times G}\times\Div^1 \arrow[d, "a"] \\
\Bun_{G}\times\Div^1                        & \Hck_{G}^{\leq z^*\chi} \arrow[l, "h_1^{z^*\chi}"'] \arrow[r, "h_2^{z^*\chi}"]                             & \Bun_G\times\Div^1                         
\end{tikzcd}
    \end{equation*}
    Now the upper left hand square is cartesian, as \(\Bun_D\times\Div^1\cong\Hck_D^{\leq\chi}\), so the lower left hand square is also cartesian and so is the left rectangle.
    Observe that the support of \(S_{\chi'}\) is contained in \(\Hck_D^{\leq\chi}\times\Hck_G^{\leq 0}\), thus we may compute Hecke operators in using this restricted diagram. 
    We get
    \begin{align*}
        h_{2!}a_{H!}(h_1'^*e'_!A \otimes S_{\chi'})&\cong h_{2!}^{z^*\chi}\bar{a}_{H!}(\bar{h}_1'^*e'_!A \otimes S_{\chi'})\\
        &\cong h_{2!}^{z^*\chi}\bar{a}_{H!}(\bar{e}'_{H!}\bar{h}_1^{0*}A \otimes S_{\chi'})\\
        &\cong h_{2!}^{z^*\chi}\bar{a}_{H!}\bar{e}'_{H!}(\bar{h}_1^{0*}A \otimes \bar{e}_{H}'^*S_{\chi'})
    \end{align*}
    Using compatibility of geometric Satake with products and its explit description for tori, we compute that \(\bar{e}_{H}'^*S_{\chi'}=\Lambda\).
    Continuing, we get
    \begin{align*}
        h_{2!}^{z^*\chi}\bar{a}_{H!}\bar{e}'_{H!}(\bar{h}_1^{0*}A \otimes \bar{e}_{H}'^*S_{\chi'})&\cong h_{2!}^{z^*\chi}\bar{a}_{H!}\bar{e}'_{H!}(\bar{h}_1^{0*}A)\\
        &\cong h_{2!}^{z^*\chi}h_{1}^{z^*\chi}A.
    \end{align*}
    Following the proof of \cite[Theorem VI.11.1.]{geometrization}, we can compute that \(S_{z^*\chi}\) is nothing but the constant sheaf on \(\Hck_{G}^{\leq z^*\chi}\), so we finally obtain
    \begin{equation*}
        h_2^{z^*\chi}h_1^{z^*\chi}A\cong h_{2!}(h_1^*A\otimes S_{z^*\chi})\cong z^*\oo(\chi)*A.
    \end{equation*}
    Let us remark that since \(\chi\) is assumed to be irreducible, functoriality in \(\chi\) means nothing more than compatiblity with endomorphisms of \(\chi\), which boils down to \(\Lambda\)-linearity.
    This is clear.
\end{proof}
\begin{lemma}\label{lem: spectral and whittaker endomorphism action agree}
    The map \(\Exc(D)\to\End(\id_{\D(\Bun_G)})=\End(\mathcal{W}_D\star -)\) induced by the map \(\Exc(D)\to\End(\mathcal{W}_D)\) agrees with the composite \(\Exc(D)\to\Exc(G)\to\End(\id_{\D(\Bun_G)})\).
\end{lemma}
\begin{proof}
    The Hecke operators \(z^*\oo(\chi)*\) of the previous lemma define a map \(\Exc(D)\to\End(\id_{\D(\Bun_G)})\), which is precisely given by the composite \(\Exc(D)\to\Exc(G)\to\End(\id_{\D(\Bun_G)})\).
    Meanwhile, via the endofunctors \((\oo(\chi)*\mathcal{W}_D)\star -\) we also obtain a map \(\Exc(D)\to\End(\id_{\D(\Bun_G)})\).
    Here the action of the excursion operators are given by the action of the excursion operators on \(\mathcal{W}_D\), then applying \(\mathcal{W}_D\star-\).
    It follows that this action is induced by the map \(\Exc(D)\to\End(\mathcal{W}_D)\) and the equivalence \(\id_{\D(\Bun_G)}\simeq\mathcal{W}_D\star-\)

\end{proof}
\begin{corollary}\label{cor: compatible on coarse moduli}
    The action of \(\QCoh(\locsyscoarse{W_E}{\check{D}})\) induced via restricting the spectral action along the map  \(\QCoh(\locsyscoarse{W_E}{\check{D}})\to\QCoh(\locsys{W_E}{\check{D}})\) agrees with the convolution action of \(\D(\Bun_{D}^0)\) via the isomorphism of \cref{lem: cartier dual BunT QCoh(ParT)}.
\end{corollary}
\begin{proof}
    Fix a compact open subgroup \(K\subset D(E)\).
    This induces an open substack \(\locsyscoarsesub{W_E}{\check{D}}{K}\) and a subcategory \(\D^K(\Bun_G)^\omega\subset\D(\Bun_G)^\omega\) as in \cite[Section IX.5.]{geometrization}.
    It follows from \cref{lem: spectral and whittaker endomorphism action agree} that convolution action of \(\D(\Bun_D^0)^\omega\) preserves \(\D^K(\Bun_G)^\omega\) and it factors through \(\D(\Bun_D^0)\to\D([*/(D(E)/K)])\), as \(\D^K(\Bun_G)^\omega\) is equivalently described as the subcategory of objects \(A\) such that the natural map \(\mathcal{W}_D\star A\to e_{K}\mathcal{W}_D\star A\) is an equivalence.
    Here \(e_K\) is the projector from \(D(E)\)-representations to \(D(E)\)-representations with \(K\)-fixed vector.
    Then \cref{lem: spectral and whittaker endomorphism action agree} and \cref{lem: cartier dual BunT QCoh(ParT)} tells us that the following diagram commutes
    \begin{equation*}
        \begin{tikzcd}
\QCoh(\locsyscoarse{W_E}{\check{D}})^\omega \arrow[rr, "\cong"] \arrow[dd] \arrow[rd] &                                      & \D(\Bun_D^0)^\omega \arrow[dd] \arrow[ld] \\
                                                                            & \QCoh(\Exc(D))^\omega \arrow[ld] \arrow[rd] &                                    \\
{\QCoh(\locsyscoarse{W_E}{\check{D}}^{K})^\omega} \arrow[rd] \arrow[rr, "\cong"]     &                                      & {\D([*/(D(E)/K)])^\omega.} \arrow[ld]       \\
                                                                            & \End(\D^K(\Bun_G)^\omega)                   &                                   
\end{tikzcd}
    \end{equation*}
    Passing to the limit over \(K\) finishes the proof.
\end{proof}
\begin{proof}[Proof of \cref{thm: geometric convolution action is spectral action}]
    Following \cref{rem: enough to give incoherent actions}, it suffies to see that the geometric and spectral actions of \(\QCoh(\locsyscoarse{W_E}{\check{D}})\) and the line bundles coming from \(B\check{D}\) agree.
    This is the content of \cref{cor: compatible on coarse moduli} and \cref{lem: compatible on line bundles}.
\end{proof}

%% file: general-compatibility-results.tex
\section{General Compatibility results}
For applications, we will need the following computation.
Let us write \(G'\defined G/D\) and we write \(p\from\Bun_G\to\Bun_{G'}\) for the natural map on the automorphic side and \(z\from\locsys{W_E}{\check{G'}}\to\locsys_{\check{G}}\) on the spectral side.
\begin{lemma}\label{lem: universal !-descent BunG BunG'}
    Let \(p\from\Bun_G\to\Bun_G'\) be the natural map.
    Then \(p\) is cohomologicallly smooth and \(p^!\Lambda\cong\Lambda\), so that \(p^!\cong p^*\).
    Then \(\D\) admits universal \(!\)-descent along \(p\).
\end{lemma}
\begin{proof}
    This follows from the fact that \(p\) is a torsor for the group stack \(\Bun_D\), which is cohomologically smooth.
    To compute \(p^!\Lambda\), let \(f\from \Bun_G\to *\) and \(g\from\Bun_{G'}\to *\) be the natural maps.
    \cite[Proposition 3.18.]{imaihamann} tells us that \(f^!\Lambda\cong\Lambda\) and \(g^!\Lambda\cong\Lambda\).
    We thus compute that \(p^!\Lambda\cong p^!g^!\Lambda\cong f^!\Lambda\cong\Lambda\).
\end{proof}
\begin{corollary}\label{cor: tensor product formula D(BunG)}
    We have an equivalence of categories \(\Mod_{\Lambda}\otimes_{(\D(\Bun_D),\star)}\D(\Bun_G)\simeq\D(\Bun_{G'})\).
    Here the map \(\D(\Bun_D)\to\Mod_\Lambda\) is \(\Gamma_c\), this is a symmetric monoidal functor for the convolution tensor product structure on \(\D(\Bun_D)\).
    The composite
    \begin{equation*}
        \D(\Bun_G)\to\Mod_{\Lambda}\otimes_{(\D(\Bun_D),\star)}\D(\Bun_G)\simeq\D(\Bun_{G'})
    \end{equation*}
    identifies with \(p_!\).
\end{corollary}
\begin{proof}
    The tensor product is computed via the colimit
    \begin{equation*}
        \colim_{n\in\Delta^\op}\D(\Bun_D)^{\otimes n}\otimes_{\Mod_\Lambda}\D(\Bun_G).
    \end{equation*}
    The simplicial object is constructed as follows.
    The natural action \(\Bun_D\) on \(\Bun_G\) gives rise to a simplicial object in \(v\)-stacks.
    Applying the functor \(\vStk\to\Pr\) that sends morphisms \(f\) to \(f_!\) then gives rise to a simplicial diagram in \(\Pr\), which precisely encodes the action of \(!\)-convolution.
    This colimit may then be rewritten as
    \begin{equation*}
        \lim_{n\in\Delta}\D(\Bun_D)^{\otimes n}\otimes_{\Mod_\Lambda}\D(\Bun_G)
    \end{equation*}
    where the transition maps are given by \(!\)-pullback.
    By \cite[Proposition 3.2.]{motivic-geometrization}, this identifies with
    \begin{equation*}
        \lim_{n\in\Delta}\D(\Bun_{D^n}\times\Bun_G),
    \end{equation*}
    again with \(!\)-pullback as transition maps.
    This identifies with \(\Bun_{G'}\) by \cref{lem: universal !-descent BunG BunG'}.

    From the colimit description we see that the composite
    \begin{equation*}
        \D(\Bun_G)\to\Mod_{\Lambda}\otimes_{(\D(\Bun_D),\star)}\D(\Bun_G)\simeq\D(\Bun_{G'})
    \end{equation*}
    identifies with the left adjoint of the projection \(\lim_{n\in\Delta}\D(\Bun_D^n\times\Bun_G)\to\Bun_G\), this is \(p^!\), therefore the map \(\D(\Bun_G)\to\Mod_{\Lambda}\otimes_{(\D(\Bun_D),\star)}\D(\Bun_G)\simeq\D(\Bun_{G'})\) identifies with \(p_!\) as claimed.
\end{proof}
\begin{corollary}\label{cor: spectral version kunneth D BunG}
    We have an equivalence of categories
    \begin{equation*}
        \Mod_{\Lambda}\otimes_{\Ind\Perf^\qc(\locsys{W_E}{\check{D}})}\D(\Bun_G)\simeq\D(\Bun_G)
    \end{equation*}
\end{corollary}

We we can refine this result to work with perverse \(t\)-structures.
\begin{lemma}\label{lem: descent t-structure}
    We have an equivalence of categories
        \begin{align*}
        \Mod_{\Lambda,\geq 0}\otimes_{(\D(\Bun_D),\star)_\geq 0}\D(\Bun_G)_{\geq0}&\simeq\D(\Bun_{G'})_{\geq0}\\
    \end{align*}
    where we mean the perverse \(t\)-structure for sheaves on \(\D(\Bun_G)\) and the usual one on \(\Mod_{\Lambda}\).
\end{lemma}
\begin{remark}
    The tensor product, as all tensor products in this article, are formed in \(\Pr\).
    While all occuring categories are Grothendieck prestable categories, it is not clear if the inclusion of Grothendieck prestable categories into additive presentable categories preserves relative tensor products, as these computed via geometric realization.
\end{remark}
\begin{proof}
    This boils down to checking that 
    \begin{equation*}
        \colim_{n\in\Delta}\D(\Bun_{D^n}\times\Bun_{G})_{\geq0}\to\D(\Bun_{G'})_{\geq0}
    \end{equation*}
    is an equivalence.
    The transition maps are \(!\)-pushforwards and it is easy to check that they are right \(t\)-exact for the perverse \(t\)-structure in this situation, by checking the dimensions of the points.
    We are left to check that \(p_!\from\D(\Bun_{G})_{\geq0}\to\D(\Bun_{G'})_{\geq0}\) generates the target, equivalently that its right adjoint \(\tau_{\geq0}p^!\) is conservative.
    However \(\tau_{\geq0}p^!\cong\tau_{\geq0}p^*\) and \(p\) has relative dimension 0, thus \(\tau_{\geq0}p^*=p^*\) and is thus conservative.
\end{proof}
\begin{lemma}\label{lem: kunneth indperf}
    Let \(W\) be a discretization of  \(W_E/P\), where \(P\) is a pro-\(p\) open subgroup of the wild inertia of \(W_E\).
    Assume that  \(\ell\nmid \lvert\pi_0(Z(G ))\rvert\).
    We have an equivalence of symmetric monoidal categories
    \begin{equation*}
        \Ind\Perf(\locsys{W}{\check{G'}})\simeq\Ind\Perf(\locsys{W}{\check{G}})\otimes_{\Ind\Perf(\locsys{W}{\check{D}})}\Mod_\Lambda.
    \end{equation*}
    Here the functor \(\Ind\Perf(\locsys{W}{\check{D}})\to\Mod_\Lambda\) is induced by pullback along the trivial \(\check{D}\)-parameter, which induces a map \(i\from\Spec(\Lambda)\to\locsys{W}{\check{D}}\).
    This functor is evidently symmetric monoidal.
    The composite
    \begin{equation*}
        \Ind\Perf(\locsys{W}{\check{G}})\to \Ind\Perf(\locsys{W}{\check{G}})\otimes_{\Ind\Perf(\locsys{W}{\check{D}})}\Mod_\Lambda\simeq\Ind\Perf(\locsys{W}{\check{G'}})
    \end{equation*}
    identifies with the pullback along the map \(\locsys{W}{\check{G'}}\to\locsys{W}{\check{G}}\).
\end{lemma}
\begin{proof}
    We first show that this claim holds for the stacks \(Z^1(F_n,\check{G})/\check{G}\), where \(F_n\) is the group on \(n\) generators, which we equip with a morphism \(F_n\to W_E\), thus we consider \(\check{G}\)-valued cocycles with the \(F_n\)-action on \(\check{G}\) being the one induced by \(F_n\to W_E\).
    Consider the cartesian square
\begin{equation*}
    \begin{tikzcd}
{Z^1(F_n,\check{G'})/\check{G}'} \arrow[r, "\bar{y}"] \arrow[d, "\bar{f}"] & {Z^1(F_n,\check{G})/\check{G}} \arrow[d, "f"] \\
* \arrow[r, "y"]                                                           & {Z^1(F_n,\check{D})/\check{D}}               
\end{tikzcd}
\end{equation*}
    By Barr-Beck (see the proof of \cite[Lemma 3.8.]{zou2025categoricallocallanglandsmathrmgln} for details), it suffices to check that \(f^*y_*\oo\cong\bar{y}_*\bar{f}^*\oo\) and that \(\bar{y}^*\) generates \(\Ind\Perf(\locsys{W_E}{\check{G'}})\) under colimits.
    Observe that \(y\) is lci, one then computes \(y_*\oo\) is a direct sum of perfect complexes of uniformly bounded Tor-dimension and we can compute the base change in quasi-coherent sheaves instead of ind-perfect complexes.
    We now check that \(\bar{y}^*\) generates \(\Ind\Perf(Z^1(F_n,\check{G'})/\check{G}')\) under colimits.
    Under the assumption \(\ell\nmid\lvert\pi_0(Z(G))\rvert\) that we impose throughout the article, the category \(\Ind\Perf(Z^1(F_n,\check{G'})/\check{G}')\) is generated by vector bundles induced by \(\check{G'}\)-representations.
    It suffices to check that \(\Ind\Perf(B\check{G'})\) is generated by \(\Ind\Perf(B\check{G})\) under colimits.
    This follows from the fact that any \(\check{G'}\)-representation \(V\) splits of as a direct summand \((\ind_{\check{G'}}^{\check{G}}V)|_{\check{G'}}\).
    Let \(\mathcal{I}\) be the category whose objects are finite sets \(I\) together with a map \(F(I)\to W\), where \(F(I)\) is the free group on \(I\).

    Consider the diagram
    \begin{equation*}
        \begin{tikzcd}
	& {\locsys{W}{\check{G}'}} && {\locsys{W}{\check{G}}} \\
	{\lim_{I\in\mathcal{I}}\locsys{F(I)}{\check{G}'}} && {\lim_{I\in\mathcal{I}}\locsys{F(I)}{\check{G}}} \\
	& {*} && {\locsys{W}{\check{D}}.} \\
	{*=\lim_{n\in\mathcal{I}}*} && {\lim_{I\in\mathcal{I}}\locsys{F(I)}{\check{D}}}
	\arrow[from=1-2, to=1-4]
	\arrow[from=1-2, to=3-2]
	\arrow[from=1-4, to=3-4]
	\arrow[from=2-1, to=1-2]
	\arrow[from=2-1, to=2-3, crossing over]
	\arrow[from=2-1, to=4-1]
	\arrow[from=2-3, to=1-4]
	\arrow[from=3-2, to=3-4]
	\arrow[from=4-1, to=3-2]
	\arrow[from=4-1, to=4-3]
	\arrow[from=4-3, to=3-4]
    \arrow[from=2-3, to=4-3, crossing over]
\end{tikzcd}
    \end{equation*}
    The front and the back square are cartesian.
    The slanted arrows induce equivalences on \(\Ind\Perf\) and the front square turns into a pushout square upon applying \(\Ind\Perf\), as we have
    \begin{equation*}
        \Ind\Perf(\lim_{I\in\mathcal{I}}\locsys{F(I)}{\check{G}})\cong\colim_{I\in\mathcal{I}}\Ind\Perf(\locsys{F(I)}{\check{G}})
    \end{equation*}
    by \cite[Proposition X.3.4.]{geometrization} and on each term in the limit \(\Ind\Perf\) turns the square into a pushout by our previous discussion.

    For the last point, consider the following diagram, where the square is cartesian.
    \begin{equation*}
        \begin{tikzcd}
\locsys{W}{\check{G}} \arrow[rd, dashed] \arrow[rrd, bend left] \arrow[rdd, bend right] &[-20pt]                                          &                               \\[-20pt]
                                                                                      & \locsys{W}{\check{G'}} \arrow[r] \arrow[d] & \locsys{W}{\check{G}} \arrow[d] \\
                                                                                      & * \arrow[r]                              & \locsys{W}{\check{D}}          
\end{tikzcd}
    \end{equation*}
    We just computed that \(\Ind\Perf\) sends the cartesian square to a pushout square in \(\CAlg(\Pr)\), i.e. satisfies categorical K\"unneth for this square.
    In such situations, the composite
        \begin{equation*}
        \Ind\Perf(\locsys{W}{\check{G}})\to \Ind\Perf(\locsys{W}{\check{G}})\otimes_{\Ind\Perf(\locsys{W}{\check{D}})}\Mod_\Lambda\simeq\Ind\Perf(\locsys{W}{\check{G'}})
    \end{equation*}
    always identifies with pullback along the dashed map.
\end{proof}
\begin{corollary}\label{cor: kunneth for parameter stack}
    We have an equivalence of categories
    \begin{equation*}
        \Ind\Perf^\qc(\locsys{W_E}{\check{G}})\otimes_{\Ind\Perf^\qc(\locsys{W_E}{\check{D}})}\Mod_\Lambda\simeq\Ind\Perf^\qc(\locsys{W_E}{\check{G'}})
    \end{equation*}
    and also without the quasi-compact support condition.
\end{corollary}
\begin{proof}
    With the quasicompact support condition, it suffices to observe that 
    \begin{equation*}
        \colim_{P\subset W_E}\Ind\Perf(\locsys{W_E/P}{\check{G}})\simeq\Ind\Perf^\qc(\locsys{W_E/P}{\check{G}}),
    \end{equation*}
    where \(P\subset W_E\) runs over open pro-\(p\) subgroups of the wild inertia.
    For the claim without quasi-compact support, we check that the functor 
    \begin{equation*}
        p^*\from\Ind\Perf(\locsys{W_E}{\check{G'}})\to\Ind\Perf(\locsys{W_E}{\check{G}})\otimes_{\Ind\Perf(\locsys{W_E}{\check{D}})}\Mod_\Lambda
    \end{equation*}
    is fully faithful and essentially surjective.
    For fully faithfulness, it suffices to check this on perfect complexes with quasi-compact support, as every perfect complex is a direct sum of perfect complexes of quasi-compact support, and this direct sum is also a direct product.
    This easy follows from the fact that \(\locsys{W_E}{\check{G}}\) is an infinite disjoint union of quasi-compact stacks.
    This case follows from the previous discussion.
    One shows essential surjectivity in the same way.
\end{proof}
From these computation, one would expect a relationship of the full categorical conjecture between \(G\) and \(G'\).
Let us start off with the spectral action.
\begin{proposition}\label{prop: commutativity different spectral actions}
    For any object \(A\in\D(\Bun_G)\), the following diagram commutes.
    \begin{equation*}
\begin{tikzcd}
\Ind\Perf^\qc(\locsys{W_E}{\check{G}})\otimes\D(\Bun_G) \arrow[d, "z^*\otimes p_!"] \arrow[r, "-*-"] & \D(\Bun_G) \arrow[d, "p_!"] \\
\Ind\Perf^\qc(\locsys{W_E}{\check{G}'})\otimes\D(\Bun_{G'}) \arrow[r, "-*-"]                         & {\D(\Bun_{G'}),}           
\end{tikzcd}
    \end{equation*}
    where \(z\from\locsys{W_E}{\check{G}'}\to\locsys{W_E}{\check{G}}\) is the natural map.
\end{proposition}
\begin{proof}
        The datum of a map \(\Perf(\locsys{W_E}{\check{G}})\to\mathcal{C}\) is the same thing as giving functorially in finite sets \(I\) a functor \(\Rep(Q)^I\)-linear \(\Rep((\check{G}\rtimes Q)^I)\to\mathcal{C}^{BW_E^I}\), using the same proof as \cite[Theorem X.0.1.]{geometrization}.

    Thus it suffices to check that \(\pi_!T_{V}(A)\cong T_{V'}(\pi_!A)\), where \(V\in\Rep((\check{G}\rtimes Q)^I)\) and \(V'\) is the inflation to \((\check{G}'\rtimes Q)^I\).
    This is effectively the same proof as \cite[Theorem IX.6.1.]{geometrization}, let us recall the details. 
    Consider the diagram
    \begin{equation*}
        \begin{tikzcd}
                      & \Hck_G^I \arrow[d, "\zeta"] \arrow[ld, "h_1"'] \arrow[rd, "h_2"]                &                                       \\
\Bun_G \arrow[d, "p"] & \Bun_G\times_{\Bun_{G'}}\Hck_{G'}^I \arrow[l, "\bar{h}_1"] \arrow[d, "\bar{p}_H"] & (\Bun_G\times\Div^1)^I \arrow[d, "p"] \\
\Bun_{G'}             & \Hck_{G'}^I \arrow[l, "h_1'"] \arrow[r, "h_2'"']                                & (\Bun_{G'}\times\Div^1)^I.            
\end{tikzcd}
    \end{equation*}
    We compute 
    \begin{align*}
        T_{V'}(p_!A)&\cong h_{2!}'(h_1'^*p_!A\otimes S_{V'})\\
        &\cong h_{2!}'(\bar{p}_!\bar{h}_1^*A\otimes S_{V'})\\
        &\cong h_{2!}'\bar{p}_!(\bar{h}_1^*A\otimes \bar{p}_H^*S_{V'})\\
        &\cong h_{2!}'\bar{p}_!(\bar{h}_1^*A\otimes \zeta_!S_{V})\\
        &\cong h_{2!}'\bar{p}_!\zeta_!(\zeta^*\bar{h}_1^*A\otimes S_{V})\\
        &\cong p_!h_{2!}(h_1^*A\otimes S_V).
    \end{align*}
    Here we used that \(\zeta\) is locally over \(\Bun_G\) the map \(\mathrm{Gr}_G^I\to\mathrm{Gr}_{G'}^I\), to the pushforward of \(S_V\) is taken to the pullback of \(S_{V'}\) by compatiblity of geometric Satake with the map \(G\to G'\) inducing isomorphisms on adjoint groups, as in the proof of \cite[Theorem VI.11.1.]{geometrization}
\end{proof}

The goal is to prove compatibility results for \cref{thm: compatiblity full categorical conjecture} under passage from \(G\) to \(G'\), let us start off with some preparatory observation on compact objects.
\begin{lemma}\label{lem: pullback compact automorphic side}
    Assume that all primes that are not banal for \(D\) are invertible in \(\Lambda\) and that \(H^1(\Gamma,D)=0\)
    Then \(p^!\from\D(\Bun_{G'})\to\D(\Bun_{G})\) sends compact object to a sheaf whose restriction to each connected component of \(\Bun_G\) is compact.
\end{lemma}
\begin{proof}
    
    Let \(A\in\D([*/G'_{b'}(E)])^\omega\).
    It is sufficient to show that \((p^!i_{b'!}A)_\kappa\) is compact, where \((-)_\kappa\) is the direct factor corresponding to the connected component \(\kappa\).
    Note that \(p^!\cong p^*\) by \cref{lem: universal !-descent BunG BunG'}, so that \(p^!i_{b'!}A\) is \(!\)-extended from \(p^{-1}(\Bun_{G'}^{b'})\).
    This fiber is a torsor under \(\Bun_D\) and \(\pi_0(\Bun_D)\) acts simply transitively on the fiber, see \cite[Section 4.5.]{isocrystals-with-additional-structure-ii} and thus decomposes into connected components according to \(\pi_0(\Bun_D)\).
    In fact, by considering the Kottwitz point we see that each point of the fiber lives in a different connected component of \(\Bun_G\), see \cite[Proposition 4.10.]{isocrystals-with-additional-structure-ii} and its proof.
    As \(\kappa\) uniquely determines an element \(b\in p^{-1}(\Bun_{G'}^{b'})\), we need to see that \(A\) remains compact after restricting along \(G_b(E)\to G_{b'}'(E)\).
    As we assumed that \(D\) is induced, the map \(G_b(E)\to G_{b'}'(E)\) is surjective. Since we assumed that the non-banal primes for \(D\) are invertible in \(\Lambda\), it follows that \(\Lambda\in\D(D(E))^\omega\).
    Therefore, the map \([*/D(E)]\to *\) is prim, this follows from \cite[Proposition 5.3.14.]{heyer-mann}.
    As \(G_b(E)\to G_{b'}'(E)\) is surjective, the map \(q\from[*/G_b(E)]\to[*/G_{b'}'(E)]\) is isomorphic to \([*/D(E)]\to *\) after pulling back along \(*\to[*/G_{b'}'(E)]\), it follows that \(q\) is prim as well.
    Pulling back along prim maps preserves prim objects, this is \cite[Lemma 4.5.16.]{heyer-mann}, and prim objects agree with the compact objects, this is \cite[Proposition 5.3.14.]{heyer-mann}, therefore \(q^*\) preserves compact objects, as we wanted to show.
\end{proof}
\begin{remark}
    The condition on non-banal primes is necessary, as can be seen from the case \(G=D\), so \(G'=\{e\}\).
\end{remark}
We now want to analyze the spectral analogues.
First let us recall that there is a 6-functor formalism \(\IndCoh\) defined on locally almost finitely presented stacks (see \cite[Definition 2.2.1.(b)]{hansen-mann}) extending the assignment \(\Spec(A)\mapsto\Ind(\Coh(A))\), such that all relevant maps and structure maps are !-able, this is \cite[Theorem 2.9.12.]{hansen-mann}.
For smooth geometric stacks \(X\), \(\QCoh(X)=\IndCoh(X)\), this is \cite[Proposition 2.10.4.]{hansen-mann}.
Finally, for the parameter stack we have \(\IndCoh(\locsys{W_E}{\check{G}})=\Ind\Coh^\qc(\locsys{W_E}{\check{G}})\) when \(\Lambda\) is a \(\bbQ\)-algebra, this is \cite[Theorem 3.3.1.(a)]{hansen-mann}.
\begin{lemma}\label{lem: pullback compact spectral side}
    Assume that \(\ell\in\Lambda^\times\), i.e. \(\Lambda\) is a \(\bbQ_\ell\)-algebra and assume that \(H^1(\Gamma,D)=0\)
    Then \(z_*\from\Ind\Coh^\qc(\locsys{W_E}{\check{G}'})\to\Ind\Coh^\qc(\locsys{W_E}{\check{G}})\) sends a compact object to a sheaf whose isotypic summand under the \(Z(\check{G})^\Gamma\)-action on sheaves on \(\locsys{W_E}{\check{G}}\) is compact.
\end{lemma}
\begin{proof}
    Consider the maps \(*\xto{y'} B\check{D}^\Gamma\xto{x'}\locsys{W_E}{\check{D}}\) picking out the trivial character.
    These pull back to maps \(\locsys{W_E}{\check{G'}}\xto{y}\locsyssub{W_E}{\check{G'}}{\check{D}=1}\xto{x}\locsys{W_E}{\check{G}}\), the space \(\locsyssub{W_E}{\check{G'}}{\check{D}=1}\) is defined as the pullback of \(\locsys{W_E}{\check{G}}\to\locsys{W_E}{\check{D}}\) along \(x'\).
    One computes that this identifies with \(Z^1_{\check{D}=1}(W_E,\check{G}')/G'\), where \(Z^1_{\check{D}=1}(W_E,\check{G})\defined Z^1(W_E,\check{G})\times_{Z^1(W_E,\check{D})}\check{D}.\mathrm{triv}\).
    With \(\mathrm{triv}\) we mean the trival \(\check{D}\)-valued cocycle and we consider the \(\check{D}\)-orbit of it in \(Z^1(W_E,\check{D})\), which is a closed subscheme.

    As \(\Lambda\) is a \(\bbQ\)-algebra, the map \(x'\) and thus also \(x\) is a local complete intersection.
    Clearly these maps are also closed immersions.

    Thus they are cohomologically smooth and proper for the \(\IndCoh\)-formalism, this is \cite[Theorem 2.9.12.(ii)]{hansen-mann}, so that \(x_*\) identifies with \(x_!\) up to twist by a shifted line bundle, which does not matter for preserving compact objects.\footnote{In fact, one can even compute that the shifted line bundle is nothing more than \(\oo[-1]\), so we only shift. This follows from \cite[Theorem 3.3.1.(iii)]{hansen-mann}}
    By cohomological smoothness, \(x_!\) preserves prim objects, see \cite[Lemma 4.5.16.]{heyer-mann}, and prim objects are the compact objects, this is \cite[Proposition 2.10.14.]{hansen-mann}.
    We are left with analyzing the behavior of \(y_*\).
    Consider the following pullback square:

    Then it suffices to check that \(\bar{y}_*\) sends coherent sheaves to such ind-coherent sheaves, such that under the \(Z(\check{G})^\Gamma\)-action on \(Z^1_{D=1}(W_E,\check{G})\) coming from restricting the \(\check{G}\)-action, the isotypic components are coherent sheaves.
    Since \(\bar{y}\) is affine, we can consider coherent sheaves concentrated in degree 0, then using presentation it suffices to consider the structure sheaf.
    This situation is base changed from \(y'\from *\to B\check{D}^\Gamma\), thus we are done if \(Z(\check{G})^\Gamma\to\check{D}^\Gamma\) is surjective.
    This is discussed in the proof of \cite[Proposition 4.10.]{isocrystals-with-additional-structure-ii}.
\end{proof}
\begin{remark}
    Under the assumption that \(H^1(\Gamma,D)=0\), the map \(B\check{D}^\Gamma\to Z^1(W_E,\check{D})/\check{D}\) picking out the trival parameter is Gorenstein whenever the non-banal primes are invertible in \(\Lambda\). This mirrors the assumption made in \cref{lem: pullback compact automorphic side}, both are used to show that a certain morphism is prim.
    Indeed we expect the lemma to be true under this weaker assumption on \(\Lambda\), the only issue is that we do not know if \(\IndCoh(\locsys{W_E}{\check{G}})=\Ind\Coh^\qc(\locsys{W_E}{\check{G}})\) is true in this generality.
\end{remark}
We obtain the following compatibility statement.
\begin{theorem}\label{thm: compatiblity full categorical conjecture}
    Assume that \(D\) satisfies \(H^1(\Gamma,D)=0\) and all primes that are not banal for \(D\) are invertible in \(\Lambda\).
    Assume that \cref{conj: categorical conjecture}(1) holds for \(G\) for a Whittaker datum \((U,\psi)\).
    This induces a Whittaker datum for \(G'\) since there is a canonical bijection between Borel subgroups of \(G\) and \(G'\). \cref{conj: categorical conjecture}(2) holds for \(G'\) for this induced Whittaker datum.
    If \(\ell\in\Lambda^\times\), then \cref{conj: categorical conjecture}(2) for \(G\) implies \cref{conj: categorical conjecture} for \(G'\).
\end{theorem}
\begin{proof}
    Let us write \(\mathcal{W}_\bullet\) for the various sheaves on \(\Bun_{D^\bullet\times G}\) via the Whittaker datum induced by \((U,\psi)\).
    Observe that \(\Bun_U\cong[*/U(E)]\), so that \(\psi\) is naturally an object of \(\D(\Bun_U)\).
    There is a canonical map of simplicial \(v\)-stacks
    \begin{equation*}
        q_\bullet\from\const_{\Bun_U}\to\Bun_{D^\bullet\times G}
    \end{equation*}
    via functoriality of \(\Bun_{(-)}\).
    Thus \(\mathcal{W}_\bullet=q_{\bullet!}\psi\), and it enhances to a natural transformation \(\mathcal{W}_\bullet\from\const_*\to\D(\Bun_{D^\bullet\times G})\). %
    Writing \(\mathcal{W}_{-1}\) for the sheaf on \(\Bun_{G'}\) induced by \((U,\psi)\) we thus obtain using \cref{prop: commutativity different spectral actions} that the following augmented simplicial diagram commutes:
    \begin{equation}\label{eq: colimit for tensoring spectral action}
                \begin{tikzcd}
        {\dots} \arrow[r, shift left=2] \arrow[r, shift right=2] \arrow[r] & \Ind\Perf^\qc(\locsys{W_E}{\check{D}\times\check{G}}) \arrow[r, shift left] \arrow[r, shift right] \arrow[d, "-*\mathcal{W}_1"] & \Ind\Perf^\qc(\locsys{W_E}{\check{G}}) \arrow[r] \arrow[d, "-*\mathcal{W}_0"] & \Ind\Perf^\qc(\locsys{W_E}{\check{G'}}) \arrow[d, "-*\mathcal{W}_{-1}"] \\
        {\dots} \arrow[r, shift left=2] \arrow[r, shift right=2] \arrow[r] & \D(\Bun_{D\times G}) \arrow[r, shift left] \arrow[r, shift right]                                                         & \D(\Bun_G) \arrow[r]                                                    & \D(\Bun_{G'}).                                                 
        \end{tikzcd}
    \end{equation}
    In fact the augmented simplicial objects are colimit diagrams of their underlying cosimplicial object.
    Passing to right adjoints, we obtain the following diagram of cosimplicial objects, which are in fact limit diagrams of their underlying cosimplicial object:
    \begin{equation*}
        \begin{tikzcd}
{\dots} & \Ind\Perf^\qc(\locsys{W_E}{\check{D}\times\check{G}}) \arrow[l, shift right=2] \arrow[l, shift left=2] \arrow[l] & \Ind\Perf^\qc(\locsys{W_E}{\check{G}}) \arrow[l, shift right] \arrow[l, shift left]     & \Ind\Perf^\qc(\locsys{W_E}{\check{G'}}) \arrow[l]           \\
{\dots} & \D(\Bun_{D\times G}) \arrow[u, "c_{\mathcal{W}_1}"'] \arrow[l, shift right=2] \arrow[l, shift left=2] \arrow[l]  & \D(\Bun_G) \arrow[u, "c_{\mathcal{W}_0}"'] \arrow[l, shift right] \arrow[l, shift left] & \D(\Bun_{G'}). \arrow[u, "c_{\mathcal{W}_{-1}}"'] \arrow[l]
\end{tikzcd}
    \end{equation*}
    Here the bottom arrows are given by \(!\)-pullback and the top arrows are given by \(*\)-pushforward.
    In particular, one cannot restrict the bottom row to compact objects, as \(!\)-pullbacks do not preserve compact objects in this situation.
    However, under the assumption that he non-banal primes of \(D\) are invertible in \(\Lambda\), the failure is not severe.
    More precisely, we claim that under this assumption \(c_{\mathcal{W}_n}\) is fully faithful on the essential image of \(\D(\Bun_{D^m\times G})^\omega\) under \(p_{[n]\to[m]}^!\), where \(p_{[m]\to[n]}\from\Bun_{D^m\times G}\to\Bun_{D^n\times G}\) is the natural map, setting \(\Bun_{D^{-1}\times G}\defined\Bun_{G'}\), whenever \([m]\to[n]\) is injective.\footnote{Observe that for visual clarity we have decided only to draw these arrows in the diagram.}
    Only considering these maps is immaterial for the limit, by \cite[Lemma 6.5.3.7.]{htt}.
    This claim is stable under isomorphisms on the source and target of \(p_{[m]\to[n]}\), so up to renaming \(D\), \(G\) and \(G'\) we may assume that we work with \(p=p_{[-1]\to[0]}\).
    In addition, it suffices to check this claim on a set of compact generators for \(\D(\Bun_{G'})\), so it suffices to check that \(c_{\mathcal{W}_0}\) induces an isomorphism
    \begin{equation*}
        \Hom_{\D(\Bun_G)}(p^!i_{b_1'!}A,p^!i_{b_2'!}B)\cong\Hom_{\Ind\Perf^\qc(\locsys{W_E}{\check{G}})}(c_{\mathcal{W}_0}p^!i_{b_1'!}A,c_{\mathcal{W}_0}p^!i_{b_2'!}B)
    \end{equation*}
    for \(b_1',b_2'\in\lvert\Bun_{G'}\rvert\) and \(A\in\D(\Bun_{G'}^{b_1'})^\omega\), \(B\in\D(\Bun_{G'}^{b_2'})^\omega\).
    We write 
    \begin{equation*}
        \Hom_{\D(\Bun_G)}(p^!i_{b_1'!}A,p^!i_{b_2'!}B)\cong\prod_{\kappa,\kappa'\in\pi_0(\Bun_G)}\Hom_{\D(\Bun_G)}((p^!i_{b_1'!}A)_\kappa,(p^!i_{b_2'!}B)_{\kappa'}),
    \end{equation*}
    where \((-)_\kappa\) is the direct factor corresponding to the connected component \(\kappa\).
    The functor \(c_{\mathcal{W}_0}\) preserves the decomposition via connected components on the source by \(\pi_0\Bun_G\cong\pi_1(G)_{\Gamma}\) and the decomposition via central grading by \(X^*(Z(\check{G})^\Gamma)\) under the isomorphism \(\pi_1(G)_\Gamma\cong X^*(Z(\check{G})^\Gamma)\), this is immediate from \cite[Lemma 5.3.3.]{categorical-fargues-tori}.
    It is therefore sufficient to show that \((p^!i_{b_1'!}A)_\kappa\) is compact.
    This is \cref{lem: pullback compact automorphic side}.

    Let us now turn our attention to the second part of the theorem, concerning the case when \(\ell\in\Lambda^\times\), i.e. \(\Lambda\) is a \(\bbQ_\ell\)-algebra.
    The parallel claim on the spectral side is \cref{lem: pullback compact spectral side}.
    Thus we have a diagram
    \begin{equation*}
        \begin{tikzcd}
{\dots} & \Coh^\qc(\locsys{W_E}{\check{D}\times\check{G}}) \arrow[l, shift right=2] \arrow[l, shift left=2] \arrow[l]            & \Coh^\qc(\locsys{W_E}{\check{G}}) \arrow[l, shift right] \arrow[l, shift left]                 & \mathcal{C} \arrow[l]                                              \\
{\dots} & \D(\Bun_{D\times G})^\omega \arrow[u, "\bbL_{\mathcal{W}_1}^\omega"'] \arrow[l, shift right=2] \arrow[l, shift left=2] \arrow[l] & \D(\Bun_G)^\omega \arrow[u, "\bbL_{\mathcal{W}_0}^\omega"'] \arrow[l, shift right] \arrow[l, shift left] & \D(\Bun_{G'})^\omega. \arrow[u, "\bbL_{\mathcal{W}_{-1}}^\omega"'] \arrow[l]
\end{tikzcd}
    \end{equation*}
    Here we define \(\mathcal{C}\) to be the essential image of \(c_{\mathcal{W}_{-1}}\) and \(\bbL_{\mathcal{W}_{\bullet}}^\omega\) to be \(c_{\mathcal{W}_{\bullet}}\), but whose target is restriced to the essential image.
    We warn the reader that neither the top nor the bottom row are limit diagrams.
    Passing to Ind-categories the bottom row becomes a limit diagram and since all \(\Ind(\bbL_{\mathcal{W}_{\bullet}}^\omega)\) are equivalences, the top row is also a limit diagram.
    Passing to left adjoints, we see that \(\Ind(\mathcal{C})\) is the colimit of 
    \begin{equation*}
        \begin{tikzcd}
{\dots} \arrow[r, shift left=2] \arrow[r, shift right=2] \arrow[r] & \Ind\Coh^\qc(\locsys{W_E}{\check{D}\times\check{G}}) \arrow[r, shift left] \arrow[r, shift right] & \Ind\Coh^\qc(\locsys{W_E}{\check{G}}) 
\end{tikzcd}
    \end{equation*}
    Using categorical K\"unneth, which in this situation follows from \cite[Proposition 4.6.2.]{indcoh-gaitsgory}, extended to stacks using \cite[Theorem 3.24.]{kesting2025categoricalkunnethformulasanalytic},  we see that we have
    \begin{equation*}
        \Mod_{\Lambda}\otimes_{\QCoh(\locsys{W_E}{\check{D}})}\Ind\Coh^\qc(\locsys{W_E}{\check{G}})\simeq\Ind(\mathcal{C}),
    \end{equation*}
    as the colimit is precisely the bar construction for the relative tensor product on the right hand side, using the equivalence \(\QCoh(\locsys{W_E}{\check{D}})\simeq\Ind\Coh^\qc(\locsys{W_E}{\check{D}})\).
    This gives
    \begin{equation*}
        \Ind\mathcal{C}\simeq\IndCoh(\locsys{W_E}{\check{G}'})=\Ind\Coh^\qc(\locsys{W_E}{\check{G}'})
    \end{equation*}
    as desired.
\end{proof}
Let us record the following consequence from the proof of \cref{thm: compatiblity full categorical conjecture}.
\begin{proposition}
    In the setting of the previous theorem, we have
    \begin{equation*}
        -*\mathcal{W}_{-1}\cong (-*\mathcal{W})\otimes_{\Ind\Perf^\qc(\locsys{W_E}{\check{D}})}\Mod_\Lambda,
    \end{equation*}
    where we write \(\mathcal{W}_{-1}\) for the sheaf on \(\Bun_{G'}\) induced by \((U,\psi)\)
\end{proposition}
\begin{proof}
    It follows from the fact that transformation of augmented simplicial objects \cref{eq: colimit for tensoring spectral action} is a colimit diagram of its underlying simplicial object, that the functor \(\Ind\Perf^\qc(\locsys{W_E}{\check{G}'})\to\D(\Bun_{G'})\) given by tensoring the given functor
    \begin{equation*}
        \Ind\Perf^\qc(\locsys{W_E}{\check{G}})\to\D(\Bun_G)
    \end{equation*}
    with \(\Mod_\Lambda\otimes_{\Ind\Perf^\qc(\locsys{W_E}{\check{D}})}-\) agrees with acting via the spectral action on \(\mathcal{W}_{-1}\).
\end{proof}

%% file: compatibility-langlands-shahidi.tex
\section{The case of (weakly) Langlands-Shahidi type parameters}\label{sec: (weakly) Langlands-Shahidi type}
We are sometimes interested in localizing the \(\ell\)-adic categorical Langlands conjectures over open loci of \(\locsys{W_E}{\check{G}}\).
Two very important cases are for \emph{Langlands-Shahidi type} parameters and \emph{weakly Langlands-Shahidi type} parameters.
Let us recall their definition.
\begin{definition}%
    A semisimple \(L\)-parameter \(\varphi\) with cuspidal support \((\check{M},\varphi_{\check{M}})\) is called \defword{Langlands-Shahidi type}, if \(R\Gamma(W_E,r_\ad^N\varphi_{\check{M}})\oplus R\Gamma(W_E,r_\ad^N\varphi_{\check{M}}^\vee)=0\), where \(N\) is the unipotent radical to a parabolic \(P\) whose Levi is \(M\) and \(r_\ad^N\) is the \(\check{M}\)-reprentation given by the adjoint action of \(\check{M}\) on \(\Lie(\check{N})\).
    We write \(\locsyscoarsesub{W_E}{\check{G}}{\LSt}\) and \(\locsyscoarsesub{W_E}{\check{G}}{\wLSt}\) for the locus of Langlands-Shahidi respectively weakly Langlands-Shahidi type in the coarse moduli, and \(\locsyssub{W_E}{\check{G}}{\LSt}\) and \(\locsyssub{W_E}{\check{G}}{\wLSt}\) for their preimage in the stack of parameters.
    We say it is \defword{weakly Langlands-Shahidi type} if \(H^2\) of the complex vanishes.
    We say a sheaf \(\pi\) of \(\Bun_G\) is (weakly) Langlands Shahidi type if all irreducible subquotients have parameter of (weakly) Langlands Shahidi type in case that \(\pi\) is compact in \(\D(G(E))\) and in general if \(\pi\) can be written as a colimit of (weakly) Langlands Shahidi type representation.
    We write \(\D^\LSt(\Bun_G)\) and \(\D^\wLSt(\Bun_G)\) for the subcategory of sheaves of Langlands-Shahidi type respectivley (weakly) Langlands-Shahidi type.
\end{definition}
\begin{remark}
    As discussed in \cite[Definition 2.4.]{zou2025categoricallocallanglandsmathrmgln}, one can alternatively define the locus of parameters of Langlands-Shahidi type with cuspidal support \(\check{M}\) at the level of the parameter stacks by defining that they are the image under the map \(f\from \locsyssub{W_E}{\check{M}}{\irred}\to\locsys{W_E}{\check{G}}\) of the \'etale locus of \(f\).
    As checked in \cite[Definition 2.4.]{zou2025categoricallocallanglandsmathrmgln}, this locus is open as \(L_f\) is perfect, and \'etale maps are open.
    It follows that parameters of Langlands-Shahidi type defined via this stacky approach are automatically semi-simple and that their locus is open.

    In fact, this discussion also works for weakly Langlands-Shahidi type parameters, by changing the following aspects.
    Instead of \(L_f\), one uses \(H^1(L_f)\).
    By the computation of the cotangent complex one sees that vanishing of \(H^1(L_f)\) is is equivalent to the condition defining weakly Langlands-Shahidi type.
    By Tate duality \(H^1(L_f)\) is the top cohomology group that does not vanish of \(L_f\), it follows that this is a coherent sheaf.
    In partiular the vanishing locus is open, and \(f\) restricted to this locus is smooth by \cite[Propositon 8.7.1.]{khan2026lecturesalgebraicstacks}.
    It follows that the locus of weakly Langlands-Shahidi type parameters is open as well.
\end{remark}
\begin{remark}
    Equivalently, we have equivalences of categories
    \begin{align*}
        \D^{\LSt}(\Bun_G)&\simeq\Ind\Perf(\locsyscoarsesub{W_E}{\check{G}}{\LSt})\otimes_{\Ind\Perf(\locsyscoarse{W_E}{\check{G}})}\D(\Bun_G)\\
        \shortintertext{and}
        D^{\wLSt}(\Bun_G)&\simeq\Ind\Perf(\locsyscoarsesub{W_E}{\check{G}}{\wLSt})\otimes_{\Ind\Perf(\locsyscoarse{W_E}{\check{G}})}\D(\Bun_G).
    \end{align*}
    From this presentation we see that inherits a semi-orthogonal decomposition according to \(\lvert\Bun_G\rvert\), and we see that one can define these notions for any locally closed substack of \(\Bun_G\).
    This description also makes it clear that there are localization functors \(\D(\Bun_G)\to\D^\LSt(\Bun_G)\) as well as \(\D(\Bun_G)\to\D^\wLSt(\Bun_G)\), which we will denote by \(A\mapsto A^\LSt\) and \(A\mapsto A^\wLSt\) respectively.
\end{remark}
Evidently, the definition only depends on the adjoint quotient of \(G\), in particular we obtain the following.
\begin{lemma}\label{lem: langlands-shahidi type invariant under central isogenies}
    Let \(G_1\to G_2\) be a morphism of connected reductive groups inducing an isomorphism on the adjoint quotient.
    Thus we obtain a map \(\iota\from\check{G}_2\to\check{G}_1\).
    A parameter \(\varphi\) is (weakly) Langlands-Shahidi type for \(\check{G}_2\) if and only if \(\iota\circ\varphi\) is (weakly) Langlands-Shahidi type for \(\check{G}_1\).
\end{lemma}
\begin{lemma}\label{lem: (weakly) Langlands-Shahidi type compatible under * and ! pullback}
    Then \(p^*\) and \(p^!\) restrict to functors \(\D^\LSt(\Bun_{G'})\to\D^\LSt(\Bun_{G})\), similarly for weakly Langlands-Shahidi type parameters.
\end{lemma}
\begin{proof}
    For \(p^*\), this is \cite[Theorem IX.6.1]{geometrization} and \(p^*\cong p^!\) by \cref{lem: universal !-descent BunG BunG'}.
\end{proof}
From this we deduce favorable descent properties for localized categories of sheaves on \(\Bun_G\).
\begin{lemma}\label{lem: !-descent of LSt and wLSt}
    We have an equivalence of categories
    \begin{align*}
        \Mod_{\Lambda}\otimes_{(\D(\Bun_D),\star)}\D^\LSt(\Bun_G)&\simeq\D^\LSt(\Bun_{G'})\\
        \shortintertext{and}
        \Mod_{\Lambda}\otimes_{(\D(\Bun_D),\star)}\D^\wLSt(\Bun_G)&\simeq\D^\wLSt(\Bun_{G'}).
    \end{align*}
    In particular \(p_!\) preserves the (weakly) Langlands-Shahidi type condition.
\end{lemma}
\begin{proof}
    We do the proof for \(\D^\LSt(\Bun_G)\), the case for weakly Langlands-Shahidi type is completely analogous.
    Recall that \(\D(\Bun_G)\to\D^\LSt(\Bun_G)\) is a Verdier quotient functor, so that \(\mathcal{C}\otimes\D(\Bun_G)\to\mathcal{C}\otimes\D^\LSt(\Bun_G)\) for any presentable category \(\mathcal{C}\) remains a Verdier quotient, in particular we still have a fully faithful functor \(\mathcal{C}\otimes\D^\LSt(\Bun_G)\injto\mathcal{C}\otimes\D(\Bun_G)\).
    As in the proof of \cref{cor: tensor product formula D(BunG)}, we need to check that under the limit presentation
    \begin{equation*}
        \D(\Bun_{G'})\simeq\lim_{n\in\Delta}\D(\Bun_D)^{\otimes n}\otimes\D(\Bun_G)
    \end{equation*}
    the subcategory \(\D^\LSt(\Bun_{G'})\) on the left hand side matches up with the subcategory \(\lim_{n\in\Delta}\D(\Bun_D)^{\otimes n}\otimes\D^\LSt(\Bun_G)\) on the right hand side.
    Appealing to \cref{lem: (weakly) Langlands-Shahidi type compatible under * and ! pullback}, it suffices to see that \(\lim_{n\in\Delta}\D^\LSt(\Bun_D^n\times\Bun_G)\) identifies with \(\D^\LSt(\Bun_{G'})\).
    It suffices to see that \(p^!\) detects the Langlands-Shahidi type condition, i.e. \((p^!)^{-1}(\D^\LSt(\Bun_{G'}))\subset\D^\LSt(\Bun_G)\).
    However \(p^!\cong p^*\) by \cref{lem: universal !-descent BunG BunG'} and thus the claim follows from \cite[Theorem IX.6.1.]{geometrization}
\end{proof}
\begin{lemma}
    We have an equivalence of categories
    \begin{align*}
        \Mod_{\Lambda,\geq0}\otimes_{(\D(\Bun_D)_{\geq0},\star)}\D^\LSt(\Bun_G)_{\geq0}&\simeq\D^\LSt(\Bun_{G'})_{\geq0}\\
        \shortintertext{and}
        \Mod_{\Lambda,\geq0}\otimes_{(\D(\Bun_D)_{\geq0},\star)}\D^\wLSt(\Bun_G)&\simeq\D^\wLSt(\Bun_{G'})_{\geq0}.
    \end{align*}
\end{lemma}
\begin{proof}
    This follows using the same argument as \cref{lem: descent t-structure}, using the fact that \(p_!\) and \(p^!\) preserve the (weakly) Langlands-Shahidi type condition by \cref{lem: !-descent of LSt and wLSt} and \cref{lem: (weakly) Langlands-Shahidi type compatible under * and ! pullback}.
\end{proof}
The importance of the Langlands-Shahidi type condition lies within the following conjecture.
\begin{conjecture}[{\cite[Conjecture 3.2.17.]{hamann-thesis}}]\label{conj: splitting semi-orthogonal Langlands-Shahidi}
    The semi-orthogonal decomposition on \(\D^\LSt(\Bun_G)\) splits, i.e.
    \begin{equation*}
        \D^\LSt(\Bun_G)\simeq\prod_{b\in\lvert\Bun_G\rvert}\D^{\LSt}(\Bun_G^b),
    \end{equation*}
    where the functor is induced by the collection \(i_{b!}\), equivalently \(i_{b*}\) equivalently \(i_{b\sharp}\).
    All three pushforwards are isomorphic in this case.
\end{conjecture}
From our previous discussion, we obtain the following.
\begin{theorem}\label{thm: langlands shahidi splitting descends}
    If \cref{conj: splitting semi-orthogonal Langlands-Shahidi} holds for \(G\), then it also holds for \(G'\).
\end{theorem}
\begin{proof}
    Let \(p\from\Bun_G\to\Bun_{G'}\) be the canonical map.
    Using \cref{lem: !-descent of LSt and wLSt}, we get
    \begin{align*}
        \D^\LSt(\Bun_{G'})&\simeq\Mod_{\Lambda}\otimes_{(\D(\Bun_D),\star)}\D^\LSt(\Bun_{G})\\
        &\simeq\Mod_{\Lambda}\otimes_{(\D(\Bun_D),\star)}\prod_{b\in\lvert\Bun_{G'}\rvert}\D^\LSt(p^{-1}(\Bun_{G'}^b))\\
        &\simeq\prod_{b\in\lvert\Bun_{G'}\rvert}\Mod_{\Lambda}\otimes_{(\D(\Bun_D),\star)}\D^\LSt(p^{-1}(\Bun_{G'}^b)).
    \end{align*}
    The last step used that countable products and countable coproducts agree in \(\Pr\), and that convolution with \(\D(\Bun_D)\) preserves \(\D^\LSt(p^{-1}(\Bun_{G'}^b))\).
    Since \(p^{-1}(\Bun_{G'}^b)\to\Bun_{G'}^b\) admits universal \(!\)-descent, being the base change of \(p\) which was proven to be universal \(!\)-descent in \cref{lem: universal !-descent BunG BunG'}.
    Then using the same proof as \cref{cor: tensor product formula D(BunG)}, we compute that 
    \begin{equation*}
        \Mod_{\Lambda}\otimes_{(\D(\Bun_D),\star)}\D(p^{-1}(\Bun_{G'}^b))\simeq\D(\Bun_{G'}).
    \end{equation*}
    Using the same argument as in \cref{lem: (weakly) Langlands-Shahidi type compatible under * and ! pullback}, we deduce that we also have
    \begin{equation*}
        \Mod_{\Lambda}\otimes_{(\D(\Bun_D),\star)}\D^\LSt(p^{-1}(\Bun_{G'}^b))\simeq\D^\LSt(\Bun_{G'}).
    \end{equation*}
\end{proof}
Another advantage is that Langlands-Shahidi type implies that the stabilizers in \(\locsys{W_E}{\check{G}}\) are diagonalizable in classical types.\footnote{This is not true for exceptional types. For example for \(G_2\) there exists a Langlands-Shahidi type parameter with stabilizer \(S_3\), see \cite[Lemma 2.4]{local-langlands-G2}.}
From this we gain the following observation that we will not need in the following discussion, but we find helpful for understanding \(\Ind\Perf^\qc\).
\begin{lemma}\label{lem: IndPerf is QCoh for classical types t-structures match}
    Assume that \(G\) has no almost simple factors of exceptional type or that \(\ell\in\Lambda^\times\)(so that \(\Lambda\) is a \(\bbQ_\ell\)-algebra).
    Then \(\Ind\Perf^\qc(\locsyssub{W_E}{\check{G}}{\LSt})\simeq\QCoh(\locsyssub{W_E}{\check{G}}{\LSt})\).
    Under this equivalence the good filtration \(t\)-structure matches up with the usual \(t\)-structure.
\end{lemma}
\begin{proof}
    In characteristic 0, the stack \(\locsyssub{W_E}{\check{G}}{\LSt}\) is a perfect stack, being a quotient stack of an infinite union of affine schemes by a reductive group.
    So we now assume that \(G\) has no almost simple factors of exceptional type..
    The map \(\locsyssub{W_E}{\check{G}}{\LSt}\to\locsyscoarsesub{W_E}{\check{G}}{\LSt}\) is an adequate moduli space and all stabilizers are diagonalizable, so the equivalence follows from \cite[Corollary 6.12.]{etale-local-structure-stacks} and \cite[Proposition 6.15.]{etale-local-structure-stacks}.
    For the claim about \(t\)-structures, note that vector bundles induced from Weyl-modules remain connective in the natural \(t\)-structure on \(\QCoh(\locsyssub{W_E}{\check{G}}{\LSt})\), it follows that the natural functor
    \begin{equation*}
        \Ind\Perf^\qc(\locsyssub{W_E}{\check{G}}{\LSt})\to\QCoh(\locsyssub{W_E}{\check{G}}{\LSt})
    \end{equation*}
    is right \(t\)-exact, but a right \(t\)-exact equivalence of categories is \(t\)-exact.
\end{proof}
Another advantage is that the \(\nilp\)-condition simplifies in this case. In fact, this already happens for weakly Langlands-Shahidi type parameters.
We get
\begin{lemma}\label{lem: nilp vanishes weakly Langlands-Shahidi type}
    Let \(\varphi\) be a parameter of weakly Langlands-Shahidi type valued in an algebraically closed field \(\Lambda\)-field \(k\).
    Then \(x^*_\varphi\Sing_{\locsys{W_E}{\check{G}}/\Lambda}\cap\mathcal{N}^*_{\check{G}}\otimes_{\Lambda}k=0\).
    Here \(\mathcal{N}_{\check{G}}^*\) denotes the nilpotent cone inside \(\check{\mathfrak{g}}^*\).
\end{lemma}
\begin{proof}
    We know that the following square is a fiber product diagram by \cite[Lemma 2.4.4.]{singsupp}, indeed the definition of weakly Langlands-Shahidi type is rigged to make this true.
    \begin{equation*}
        \begin{tikzcd}
\mathrm{Sing}_{\locsys{W_E}{\check{M}}{\wLSt}} \arrow[d] \arrow[r]   & \locsyssub{W_E}{\check{M}}{\wLSt} \arrow[d, "f"] \\
\mathrm{Sing}_{\locsyssub{W_E}{\check{G}}{\wLSt}} \arrow[r] & \locsyssub{W_E}{\check{G}}{\wLSt}          
\end{tikzcd}
    \end{equation*}
    It follows that the composite
    \begin{equation*}
        \mathrm{Sing}_{\locsys{W_E}{\check{M}}{\wLSt}} \to\locsyssub{W_E}{\check{M}}{\wLSt}\times_{B\check{G}}\Lie{\check{G}}/\check{G}
    \end{equation*}
     pulled back from \(\mathrm{Sing}_{\locsys{W_E}{\check{M}}}{\wLSt}\to\locsyssub{W_E}{\check{G}}{\wLSt}\) factors through
     \begin{equation*}
        \locsyssub{W_E}{\check{M}}{\wLSt}\times_{B\check{M}}{\mathfrak{m}^*}/\check{M}\to\locsyssub{W_E}{\check{M}}{\wLSt}\times_{B\check{G}}\Lie{\check{G}}/\check{G}.
     \end{equation*}
    Under \(f\)  vectors \(v\in x^*_\varphi\Sing_{\locsys{W_E}{\check{G}}/\Lambda}\cap\mathcal{N}^*_{\check{G}}\otimes_{\Lambda}k\) pull back to vectors \(f^*v\in x^*_{\varphi_{\check{M}}}\Sing_{\locsys{W_E}{\check{M}}/\Lambda}\cap\mathcal{N}^*_{\check{M}}\otimes_{\Lambda}k\), where \(\varphi_{\check{M}}\) is a lift of \(\varphi\) to \(\check{M}\) that is irreducible.
    The latter space vanishes by \cite[Lemma 2.8.]{zou2025categoricallocallanglandsmathrmgln}.
    As \(f^*\) induces an isomorphism on the singular cone, it is conservative and it follows that \(v=0\) too, as desired.
\end{proof}
In particular, the \(\ell\)-adic categorical Langlands conjecture simplifies to the following form.
\begin{conjecture}\label{conj: langlands shahidi type conjecture}
    Fix a Whittaker datum \((U,\psi)\), giving rise to a sheaf \(\mathcal{W}\defined i_{1!}\cind_{U(E)}^{G(E)}\psi\).
    Then there is an equivalence of categories
    \begin{equation*}
        \Ind\Perf^\qc(\locsyssub{W_E}{\check{G}}{\LSt})\simeq\D^\LSt(\Bun_G)
    \end{equation*}
    with the following properties:
    \begin{enumerate}
        \item It maps the structure sheaf \(\oo\) to \(\mathcal{W}^\wLSt\).
        \item It is linear for the spectral action of \(\Perf(\locsyssub{W_E}{\check{G}}{\wLSt})\) on both sides (In particular, this pins down the functor together with the previous point).
        \item It is \(t\)-exact for the good filtration \(t\)-structure on the left hand side and the perverse \(t\)-structure on the right hand side.
    \end{enumerate}
    The same holds true if we replace \(\wLSt\) everywhere with \(\LSt\).
\end{conjecture}
\begin{remark}\label{rem: discussion tilting}
    Point (3) is equivalent to the claim that Hecke operators attached to Weyl modules or more generally \(\check{G}\)-representations with good filtration are right \(t\)-exact.
    For tilting modules, this appears in \cite[Conjecture 6.4.]{torsionvanishing}, at least for the subcategory localized at a single parameter.
    If \(V\) is a tilting module, then both \(V\) and \(V^\vee\) admit a good filtration, so Point (3) implies that the Hecke action for tilting modules is \(t\)-exact, as conjectured there.
    Let us also remark that there is a conjecture that local Shimura varieties at finite level are Stein, see \cite[Conjecture 1.10.]{t-exact}, this would imply \(t\)-exactness for Hecke operators attached to minuscule weights, and by tensoring minuscule representation one would get \(t\)-exactness for all Weyl modules whose highest weight is a direct sum of minuscule representations.
    This gives all Weyl modules except for groups that contain simple factors of type \(G_2\), \(F_4\) or \(E_8\).

    By \cref{lem: IndPerf is QCoh for classical types t-structures match}, often we can just write \(\QCoh\) and take the usual \(t\)-structure.
    Point (3) cannot be expected to hold without imposing the Langlands-Shahidi type assumption.
\end{remark}
We get the following compatibility result.
\begin{theorem}\label{thm: Langlands-Shahidi type compatibility}
    Assume that \(H^1(\Gamma,D)=0\).
    Assume that \(\ell\nmid|\pi_0(Z(G))|\).
    Let \((U,\psi)\) be a Whittaker datum for \(G'\), this induces a Whittaker datum for \(G\) by observing that the unipotent radical of the Borel does not change under passage from \(G\) to \(G'\).
    We will give this Whittaker datum for \(G\) the same name.
    If \cref{conj: langlands shahidi type conjecture} holds for \(G\) for a Whittaker datum \((U,\psi)\), then it also holds for \(G'\) for \((U,\psi)\).
\end{theorem}
\begin{proof}
    For the first two points, the proof is essentially the same as in \cref{thm: compatiblity full categorical conjecture}, with the simplification that \(a_\psi\) is already an equivalence and thus we can omit the discussion about the behavior of compact objects under \(p^!\).
    This is why there are no assumptions on \(\ell\) appearing in the statement.

    Regarding the \(t\)-exactness, we proceed as follows.
    We know from \cref{lem: !-descent of LSt and wLSt}, that
    \begin{equation*}
        \colim_{n\in\Delta}\D^\wLSt(\Bun_D^n\times\Bun_G)\to\D^\wLSt(\Bun_{G'})
    \end{equation*}
    is an equivalence of categories.
    All transition maps are \(!\)-pushforwards, and one checks they are right \(t\)-exact for the perverse \(t\)-structure.
    It follows that we have
    \begin{equation*}
        \colim_{n\in\Delta}\D^\wLSt(\Bun_D^n\times\Bun_G)_{\geq 0}\to\D^\wLSt(\Bun_{G'})_{\geq 0}
    \end{equation*}
    is an equivalence of categories, in particular \(p_!\) generates the target.
    Let \(V'\) be a Weyl module for \(\check{G'}\).
    It suffices to see that its action is right \(t\)-exact.
    Then \(\ind_{\check{G'}}^{\check{G}}V'\) admits a good filtration, so we know its action on \(\D^\wLSt(\Bun_{G'})\) is right \(t\)-exact.
    We find that \(T_{\ind_{\check{G'}}^{\check{G}}V'}p_!A\cong p_!T_{\ind_{\check{G'}}^{\check{G}}V'|_{\check{G}'}}A\) for all \(A\) as in the proof of Point (2) in \cref{thm: Langlands-Shahidi type compatibility}.
    If \(A\) is connective in the perverse \(t\)-structure, then we deduce that the right hand side and therefore the left hand side is connective.
    Since \(V'\) occurs as a direct summand of \(\ind_{\check{G'}}^{\check{G}}V'|_{\check{G}'}\) we deduce that \(T_{V'}p_!A\) is connective in the perverse \(t\)-structure if \(A\) is.
    Since \(p_!\) generates the connective part of the perverse \(t\)-structure it follows that \(T_{V'}\) is right \(t\)-exact.
\end{proof}

\begin{remark}
    This all falls into a bigger paradigm.
    Let \(\D^\temp(\Bun_G)\) denote the essential image of \(-*\mathcal{W}_{\psi}\).\footnote{Here \(\D^\temp\) stands for ``tempered'' in analogy with the usual geometric Langlands correspondence. However this notion of temperedness has little to do with the notion of temperedness in the classical Langlands correspondence. For example the constant is tempered in the classical sense on each stratum, but is orthogonal to \(\D^\temp(\Bun_G)\) categorically. Hence it is as un-tempered as possible in that meaning.}
    Then there ought to be a localization functor \(\D(\Bun_G)\to\D^\temp(\Bun_G)\).
    We define a \(t\)-structure on \(\D^\temp(\Bun_G)\) to be the unique one that makes the localization functor \(t\)-exact.
    Then we expect \(\Ind\Perf^\qc(\locsys{W_E}{\check{G}})\simeq\D^\temp(\Bun_G)\) to be \(t\)-exact.
    In the usual geometric Langlands correspondence, this is \cite[Theorem F]{non-vanishig-whittaker}.
    Using the same arguments as before, a \(t\)-exact equivalence \(\Ind\Perf^\qc(\locsys{W_E}{\check{G}})\simeq\D^\temp(\Bun_G)\) implies the same \(t\)-exact equivalence for \(G'\).
    This does not quite imply \cref{thm: Langlands-Shahidi type compatibility}, as we don't know  \(\D^{\wLSt}(\Bun_G)\subset\D^\temp(\Bun_G)\).
    Then \cref{lem: nilp vanishes weakly Langlands-Shahidi type} is the spectral analogue of the claim \(\D^{\wLSt}(\Bun_G)\subset\D^\temp(\Bun_G)\).

    Let us also remark that there is a way to define \(\D^\temp(\Bun_G)\) which makes independence of the choice of Whittaker datum clear, at least in characteristic 0.
    For the rest of the remark, let us work with \(\Lambda=\overline{\bbQ_{\ell}}\).
    Namely, consider the derived Satake equivalence \(\Ind\D^{\ula}(\localhck_{G,C})=\Ind\Coh(*\times_{\check{\mathfrak{g}}}*/\check{G})\) of \cite[Theorem 1.2.]{bando2026derivedgeometricsatakeequivalence}
    Here we implictly choose a divisor \(C\to\Div^1\) on the Fargues-Fontaine curve.
    As the action of \(\D^{\ula}(\localhck_{G,C})\) preserves compact objects of \(\D(\Bun_{G,C})\simeq\D(\Bun_G)\)\footnote{Acting with a ULA sheaf \(V\) admits a left and right adjoint given by asking with \(\mathrm{sw}^*\bbD V\), where \(\bbD\) is the Verdier dual and \(\mathrm{sw}\from\localhck_G\to\localhck_G\) exchanges the two bundles.}, we have an action of \(\Ind\D^\ula(\localhck_{G,C})\) on \(\D(\Bun_G)\).
    Then \(\D^\temp(\Bun_G)\) is the subcategory of \(\D(\Bun_G)\) such that the action of \(\Ind\D^\ula(\localhck_{G,C})\simeq\Ind\Coh(*\times_{\check{\mathfrak{g}}}*/\check{G})\) factors through \(\QCoh(*\times_{\check{\mathfrak{g}}}*/\check{G})\).
    This might a priori depend on a choice of map \(C\to\Div^1\), but using the equivalence \(\Ind\D^{\ula}(\localhck_{G,\Div^1})\simeq\Ind\Coh(*\times_{\check{\mathfrak{g}}(-1)}*/\check{G})\) as categories with \(W_E\)-action from \cite[Theorem 1.1.]{bando2026derivedgeometricsatakeequivalence}, one sees that the choice is immaterial.
\end{remark}

%% file: application-pgln.tex
\section{Application to \texorpdfstring{\(\PGL_n\)}{PGLn}}
We will be able to resolve all \cite{beijing-notes} in the case of \(\PGL_n\) about generous parameters by reduction to \(\GL_n\), where this is done in \cite[Section 8]{zou2025categoricallocallanglandsmathrmgln}.
\begin{theorem}
   We have a direct product decomposition
   \begin{equation*}
    \D^\LSt(\Bun_{\PGL_n})\simeq\prod_{b\in B(\PGL_n)}\D^\LSt(\Bun_{\PGL_n}^b).
   \end{equation*}
   The decomposition is via excision and the \(!\) and \(*\) functors agree.
   For \(\Lambda\) a flat \(\bbZ_\ell\)-algebra or a \(\bbZ_\ell\)-field Hecke operators are \(t\)-exact. 
   The analogue for categories localized at a single parameter of Langlands-Shahidi type holds too.
\end{theorem}
\begin{proof}
    This follows from \cref{thm: langlands shahidi splitting descends} and \cite[Theorem 8.2]{zou2025categoricallocallanglandsmathrmgln}.
\end{proof}
\begin{theorem}
    Fix a Whittaker datum \((U,\psi)\) for \(\PGL_n\), giving rise to a sheaf \(\mathcal{W}\in\D(\Bun_{\PGL_n})\).
    We have an equivalence of categories
    \begin{equation*}
        \QCoh(\locsyssub{W_E}{\SL_n}{\LSt})\simeq\D^\LSt(\Bun_{\PGL_n})
    \end{equation*}
    sending \(\oo\) to \(\mathcal{W}^\LSt\) and linear for the spectral action of \(\Ind\Perf(\locsyssub{W_E}{\SL_n}{\LSt})\) on both sides.
    If \(\Lambda\) is additionally a flat \(\bbZ_\ell\)-algebra or a \(\bbZ_\ell\)-field, this is also an \(t\)-exact.
\end{theorem}
\begin{proof}
    This follows from \cref{thm: Langlands-Shahidi type compatibility} and \cite[Theorem 6.2.]{zou2025categoricallocallanglandsmathrmgln}
\end{proof}